\documentclass[reqno,letterpaper,11pt]{article}

\usepackage{hyperref}

\usepackage{times}
\usepackage[mathcal]{euscript}

\usepackage{amsmath}
\usepackage{amsfonts,amssymb,amsmath,latexsym,wasysym,mathrsfs,bbm,stmaryrd}
\usepackage[all]{xy}
\usepackage[usenames]{color}
\usepackage{epsfig}

\usepackage{amsthm}
\usepackage{thmtools}
\usepackage{thm-restate}
\declaretheorem[numberwithin=section]{theorem}
\declaretheorem[sibling=theorem]{proposition}

\declaretheorem[sibling=theorem]{corollary}
\declaretheorem[sibling=theorem]{lemma}

\declaretheorem[sibling=theorem,style=remark]{remark}

\numberwithin{equation}{section} 

\def\R{\mathbb R}
\def\C{\mathbb C}

\def\N{\mathbb N}

\def\1{\mathbbm 1}

\def\a{\alpha}

\def\H{\mathcal{H}}

\def\jj{\mx{j}}

\def\del{\partial}

\newcommand{\mx}[1]{\mathbf{#1}}

\begin{document}

\title{On Sharp Constants for Dual Segal--Bargmann $L^p$ Spaces}
\author{William E.\ Gryc \\
Department of Mathematics and Computer Science \\
Muhlenberg College \\
Allentown, PA 18104-5586 \\
\texttt{wgryc@muhlenberg.edu}
\and
Todd Kemp\thanks{Supported by NSF CAREER Award DMS-1254807} \\
Department of Mathematics \\
University of California, San Diego \\
La Jolla, CA 92093-0112 \\
\texttt{tkemp@math.ucsd.edu}
}

\date{\today}

\maketitle

\begin{abstract} We study dilated holomorphic $L^p$ space of Gaussian measures over $\C^n$, denoted $\mathcal{H}_{p,\alpha}^n$ with variance scaling parameter $\alpha>0$.  The duality relations $(\mathcal{H}_{p,\alpha}^n)^\ast \cong \mathcal{H}_{p',\alpha}$ hold with $\frac{1}{p}+\frac{1}{p'}=1$, but not isometrically.  We identify the sharp lower constant comparing the norms on $\mathcal{H}_{p',\alpha}$ and $(\mathcal{H}_{p,\alpha}^n)^\ast$, and provide upper and lower bounds on the sharp upper constant.  We prove several suggestive partial results on the sharpness of the upper constant.  One of these partial results leads to a sharp bound on each Taylor coefficient of a function in the Fock space for $n=1$.\end{abstract}

\tableofcontents

\section{Introduction}

This paper is concerned with the holomorphic $L^p$ spaces associated to Gaussian measures on $\C^n$.  In the case $p=2$, such spaces are often called Segal--Bargmann spaces \cite{Hall} or Fock spaces \cite{Zhu}.  They are core examples in the theory of holomorphic reproducing kernel Hilbert spaces, with connections to quantum field theory, stochastic analysis, and beyond.  The scaling of duality between these holomorphic $L^p$-spaces is still not fully understood; this paper presents some new sharp results, and new puzzles about these dual norms.

To fix notation, let $\alpha>0$, $n\in\N$, and let $\gamma^n_{\alpha}$ denote the following Gaussian probability measure on $\C^n$:
\[ \gamma^n_{\alpha}(dz) = \left(\frac{\alpha}{\pi}\right)^n e^{-\alpha|z|^2}\,\lambda^n(dz), \]
where $\lambda^n$ is the Lebesgue measure on $\C^n$.  The spaces considered in this paper are of the form $L^p_{hol}(\gamma^n_\alpha)$ for $1\le p<\infty$ and some $\alpha>0$, the subspaces of the full $L^p(\gamma^n_\alpha)$-spaces consisting of holomorphic functions.  These are Banach spaces in the usual $L^p$-norm.  However, as discovered by Sj\"ogren \cite{Sjogren} and proved as \cite[Proposition 1.5]{GrycKemp2011}, in this scaling, with $\frac{1}{p}+\frac{1}{p'}=1$ as usual, $L^p_{hol}(\gamma^n_\alpha)$ and $L^{p'}(\gamma^n_\alpha)$ are {\em not} dual to each other when $p\ne 2$.  It was shown by Janson, Peetre, and Rochberg \cite{JPR} that the correct scaling requires the parameter $\alpha$ to dilate with $p$.  That is, we define the {\em dilated holomorphic $L^p$ space} as
\begin{equation} \label{e.Hpa} \mathcal{H}_{p,\alpha}^n \equiv L^p_{hol}(\gamma^n_{\alpha p/2}) = \left\{f\in\mathrm{Hol}(\C^n)\colon \int \left|f(z)e^{-\alpha|z|^2/2}\right|^p\,\lambda^n(dz) < \infty\right\}, \end{equation}
with norm
\begin{equation} \label{e.Hpa.norm} \|h\|_{p,\alpha} = \|h\|_{\mathcal{H}^n_{p,\alpha}} \equiv \left(\int |h|^p \,d\gamma^n_{\alpha p/2}\right)^{1/p}. \end{equation}
Similarly, if $\Lambda\in (\mathcal{H}^n_{p,\alpha})^{\ast}$ is a bounded linear functional, denote its dual norm by
\begin{equation} \label{e.Hpa*.norm} \|\Lambda\|_{p,\alpha}^\ast = \|\Lambda\|_{(\mathcal{H}^n_{p,\alpha})^\ast} \equiv \sup_{g\in\mathcal{H}^n_{p,\alpha}\setminus\{0\}}\!\!\! \frac{ |\Lambda(g)| }{\;\;\;\|g\|_{p,\alpha} }. \end{equation}
(We de-emphasize the $n$-dependence of the norms $\|\cdot\|_{p,\alpha}$ and $\|\cdot\|_{p,\alpha}^\ast$; it will always be clear from context.)  It was shown in \cite{JPR} that $\mathcal{H}^n_{p,\alpha}$ and $\mathcal{H}^n_{p'\!,\alpha}$ {\em are} dual spaces for $1<p<\infty$.  One of the two main theorems of the present authors' paper \cite{GrycKemp2011} was the following estimate on the sharp constants of comparison for the dual norms.

\begin{theorem}[Theorem 1.2 in \cite{GrycKemp2011}] Let $1<p<\infty$, and $\frac{1}{p}+\frac{1}{p'}=1$.  Define the constant $C_p$ by
\begin{equation} \label{e.Cp} C_p \equiv 2\frac{1}{p^{1/p}}\frac{1}{{p'}^{1/p'}}. \end{equation}
Let $n\in\N$ and $\alpha>0$. Define $\langle f,g\rangle_\alpha = \int_{\C^n} f\overline{g}\,d\gamma^n_\alpha$.   Then for any $h\in\mathcal{H}^n_{p'\!,\alpha}$,
\begin{equation} \label{e.norm.equiv} \|h\|_{p'\!,\alpha} \le \| \langle \cdot,h\rangle_\alpha\|_{p,\alpha}^\ast \le C_p^n  \|h\|_{p'\!,\alpha}. \end{equation}
\end{theorem}
Presently, we are interested in the sharpness of the inequalities in \eqref{e.norm.equiv}.  In fact, the first inequality is sharp, and this yields a new concise proof of a pointwise bound for the space $\mathcal{H}^n_{p,\alpha}$.

\begin{theorem} \label{t.1} Let $1<p<\infty$, $\frac{1}{p}+\frac{1}{p'}=1$, and $\alpha>0$.  Then
\[ \inf_{h\in\mathcal{H}^n_{p'\!,\alpha}\setminus\{0\}} \frac{  \| \langle \cdot,h\rangle_\alpha\|_{p,\alpha}^\ast}{ \|h\|_{p'\!,\alpha} } = 1. \]
It follows that, for any $z\in\C^n$, and any $g\in\mathcal{H}^n_{p,\alpha}$,
\begin{equation} \label{e.ptwise.bd} |g(z)| \le e^{\frac{\alpha}{2}|z|^2}\|g\|_{p,\alpha}. \end{equation}
\end{theorem}

\begin{remark} The bound \eqref{e.ptwise.bd} is well-known; it can be found, for example, as \cite[Theorem 2.8]{Zhu}.  In fact, it is common to define a supremum norm on holomorphic functions $g$ as
$$\|g\|_{\infty,\alpha}=\sup_{z\in\C^n} g(z)e^{-\frac{\alpha}{2}|z|^2},$$
in which case \eqref{e.ptwise.bd} can be elegantly rewritten as $\|g\|_{\infty,\alpha}\leq\|g\|_{p,\alpha}$ for $1<p<\infty$.
\end{remark}

That the first inequality in \eqref{e.norm.equiv} should hold sharply is natural to expect from the method of proof given in \cite{GrycKemp2011}.  Indeed, it is instructive to write \eqref{e.norm.equiv} in the alternate form proven in our first paper.  Note that $\mathcal{H}^n_{2,\alpha}$ is a closed subspace of $L^2(\gamma^n_\alpha)$; let $P_\alpha^n\colon L^2(\gamma^n_\alpha)\to \mathcal{H}^n_{2,\alpha}$ denote the orthogonal projection. In fact, $P^n_\alpha$ is an integral operator that is bounded from $L^p(\gamma^n_{\alpha p/2})$ to $\mathcal{H}^n_{p,\alpha}$ for all $1<p<\infty$, as was originally shown in \cite{JPR}.  Denote by $\|P_\alpha^n\|_{p\to p}\equiv \|P_\alpha^n\colon L^p(\gamma^n_\alpha)\to \mathcal{H}^n_{p,\alpha}\|$.  In \cite[Lemma 1.18]{GrycKemp2011}, we proved that
\begin{equation} \label{e.alt.form} \frac{1}{\|P_\alpha^n\|_{p\to p} }\|h\|_{p'\!,\alpha} \le  \frac{1}{C_p^n}\| \langle \cdot,h\rangle_\alpha\|_{p,\alpha}^\ast \le  \|h\|_{p'\!,\alpha}.   \end{equation}
The $1/C_p^n$ in the middle term comes from the global geometry underlying these spaces.  Note from \eqref{e.Hpa} that $\mathcal{H}^n_{p,\alpha}$ can be though of as consisting of ``holomorphic sections'': functions $F$ of the form $F(z) = f(z)e^{-\alpha|z|^2/2}$ for some holomorphic $f$; the integrability condition for containment in $\mathcal{H}_{p,\alpha}^n$ is then simply that $F\in L^p(\C^n,\lambda^n)$.  The factor $1/C_p^n$ then arises from the constants relating the norms $\|\cdot\|_{p,\alpha}$ and $\|\cdot\|_{p'\!,\alpha}$ to the $L^p(\C^n,\lambda^n)$- and $L^{p'}(\C^n,\lambda^n)$-norms, yielding $p^{1/p}$ and ${p'}^{1/p'}$ factors from the normalization coefficients of the measures $\gamma^n_{\alpha p/2}$ and $\gamma^n_{\alpha p'/2}$.  The first inequality in \eqref{e.alt.form} then simplifies due to the first main theorem \cite[Theorem 1.1]{GrycKemp2011}, which states that $\|P^n_\alpha\|_{p\to p} = C_p^n$. The sharpness of the first inequality in \eqref{e.norm.equiv} is indicative of the fact that the orthogonal projection $P_\alpha^n$ controls the geometry of the spaces $\mathcal{H}^n_{p,\alpha}$.

In this context, the second inequality in \eqref{e.alt.form}, and hence in \eqref{e.norm.equiv}, is simply H\"older's inequality.  In the larger spaces $L^p(\C^n,\lambda^n)$ and $L^{p'}(\C^n,\lambda^n)$ where the section spaces $\mathcal{H}^n_{p,\alpha}$ live, H\"older's inequality is, of course, sharp: if $F\in L^p(\C^n,\lambda^n)$, then the function $G = |F|^{p-2}F$ is in $L^{p'}(\C^n,\lambda^n)$ and $\|G\|_{p'} = \|F\|_p^{p/p'}$, so that $\langle F,G\rangle = \|F\|_p^p = \|F\|_p \|G\|_{p'}$.  However, the function $G$ is typically not a holomorphic section, and so it is not a surprise that the same saturation argument fails in the spaces $\mathcal{H}^n_{p,\alpha}$.  In fact, we can say more.

\begin{theorem} \label{t.2} Let $1<p<\infty$, $p\neq 2$, $\frac{1}{p}+\frac{1}{p'}=1$, and $\alpha>0$.  If $g\in \mathcal{H}^n_{p,\alpha}$ and $h\in \mathcal{H}^n_{p'\!,\alpha}$ are non-zero, then
\begin{equation} \label{e.sharp.Holder} |\langle g,h\rangle_\alpha| < C_p^n \|g\|_{p,\alpha} \|h\|_{p'\!,\alpha}. \end{equation}
\end{theorem}
Theorem \ref{t.2} asserts that H\"older's inequality is a {\em strict} inequality in the Segal-Bargmann spaces.  It is a priori possible that the inequality is nevertheless saturated by a sequence in $\mathcal{H}^n_{p,\alpha}\times \mathcal{H}^n_{p',\alpha}$, but we believe this is not the case.  Indeed, we conjecture that the second inequality in \eqref{e.norm.equiv} is not sharp.  To the question of the sharp constant, we prove the following.

\begin{theorem} \label{t.3} Let $1<p<\infty$, $\frac{1}{p}+\frac{1}{p'}=1$, $n\in\N$, and $\alpha>0$.  For $g\in\mathcal{H}^n_{p,\alpha}$ and $h\in\mathcal{H}^n_{p'\!,\alpha}$ nonzero, define
\[ \mathcal{R}_{p,\alpha}(g,h) = \frac{ |\langle g,h\rangle_\alpha | } { \|g\|_{p,\alpha} \|h\|_{p'\!,\alpha} }. \]
Then
\begin{equation} \label{e.sharp?} C_p^{n/2} \le \sup_{g\in \mathcal{H}^n_{p,\alpha}\setminus\{0\}}\; \sup_{h\in \mathcal{H}^n_{p'\!,\alpha}\setminus\{0\}} \mathcal{R}_{p,\alpha}(g,h) \le C_p^n. \end{equation}
\end{theorem}
\noindent We can exhibit sequences in $\mathcal{H}^n_{p,\alpha}\times\mathcal{H}^n_{p'\!,\alpha}$ that saturate \eqref{e.sharp.Holder} with $C_p^{n/2}$ in place of $C_p^n$, as will be demonstrated in the proof of Theorem \ref{t.3} (cf.\ \eqref{e.sharp?2}); indeed, such a saturating sequence can be built from monomials.  If the same bound could be shown to hold not only for monomials but all {\em holomorphic polynomials}, this would prove the sharpness of the $C_p^{n/2}$-bound in general (since holomorphic polynomials are dense in $\mathcal{H}^n_{p,\alpha}$).  This conjecture seems to be quite difficult to prove.  The final major results of this paper are two partial results in this direction, summarized as follows.

\begin{theorem} \label{t.4.new} Let $1<p<\infty$, $n\in\N$, and $\alpha>0$.  Let $\mathcal{M}^n$ denote the space of holomorphic monomials on $\C^n$.  Then
\[ \sup_{g\in\H^n_{p,\a}\setminus\{0\}}\sup_{h\in\mathcal{M}^n} \mathcal{R}_{p,\alpha}(g,h) =C_p^{n/2}. \]
\end{theorem}
\noindent That is: restricting one variable of $\mathcal{R}_{p,\alpha}$ to run through monomials, but letting the other run freely over $\H^n_{p,\alpha}$ yields the conjectured global maximum.

Theorem \ref{t.4.new} is proved as a corollary to Theorem \ref{lemma3} below, which has its own independently interesting corollaries.  In particular, we have the following result:

\begin{corollary}\label{t.4.5.new}
Let $1<p<\infty$ and $\a>0$. Let $f\in\mathcal{H}^1_{p',\a}$ and with the Taylor series $\sum_{k=0}^\infty a_k z^k$. Then for any $j\in\N\cup\{0\}$, we have
\begin{equation}
\|a_j z^j\|_{p,\alpha}\leq \left\|\sum_{k=0}^\infty a_k z^k\right\|_{p,\alpha}=\|f\|_{p,\alpha}\nonumber,
\end{equation}
and
\begin{equation}\label{t.4.5.new.eq1}
|a_j|\leq \frac{\left(\frac{\alpha p}{2}\right)^{j/2}}{\Gamma\left(\frac{j p}{2}+1\right)^{1/p}}\|f\|_{p,\alpha}.
\end{equation}
In either inequality above, we have equality if and only if $f(z)$ is a constant multiple of $z^j$.
\end{corollary} 
Corollary \ref{t.4.5.new} is proven as Corollary \ref{cor5n=1}. A non-sharp bound (but one independent of $j$) akin to \eqref{t.4.5.new.eq1} can be found in \cite[Ch. 2, Ex. 18]{Zhu}, and a growth condition on Taylor coefficients can be found as Corollary 5 in \cite{Tung1}.  We believe that the sharp estimate of \eqref{t.4.5.new.eq1} is new.  Finally, we prove the following partial result: restricting to {\em Gaussian}-like functions in $\H^1_{p,\alpha}$ yields the desired maximum.

\begin{theorem} \label{t.5} Let $\mathcal{G}_\alpha$ denote the space of quadratic exponential functions in $\H^1_{2,\alpha}$ (cf.\ \eqref{Galphadef}).  Then
\[ \sup_{g,h\in\mathcal{G}_\alpha} \mathcal{R}_{p,\alpha}(g,h) = \sqrt{C_{p}}. \]
\end{theorem}
\noindent Given the Gaussian nature of the spaces $\H^1_{p,\a}$, it is very natural to expect the maximum of $\mathcal{R}_{p,\alpha}$ to be achieved on ``Gaussian'' functions, in light of \cite{Lieb}, for example.  Theorem \ref{t.5} is proved as Theorem \ref{gaussianprop} below.

\section{The Sharp Lower Constant}
Our overall goal in this section is to prove Theorem \ref{t.1}.  To prove this theorem and others in the paper, many integrals involving Gaussian and exponential functions will be calculated.  Often the details of these calculations will be omitted, but are based on the following formula (cf.\ \cite{LiebLoss}):
\begin{lemma} \label{l.Gauss.int}
Let $A$ be a complex symmetric matrix, $v$ a vector in $\R^k$, and let $(\cdot,\cdot)$ denote the standard inner product on $\R^k$.  Define the function $f(x)=\exp(-(x,Ax)+2(v,x))$.  Then $f\in L^1(\R^k)$ if and only if $\Re(A)$ is positive definite, and in this case,
\begin{equation}
\int_{\R^k} e^{-(x,Ax)+2(v,x)}\, dx = \frac{\pi^{k/2}}{\sqrt{\det(A)}}e^{(v,A^{-1}v)}.
\end{equation}
\end{lemma}

Theorem \ref{t.1} has two claims: first, that 
\begin{equation}\label{t.1.1}
\inf_{h\in\mathcal{H}^n_{p'\!,\alpha}\setminus\{0\}} \frac{  \| \langle \cdot,h\rangle_\alpha\|_{p,\alpha}^\ast}{ \|h\|_{p'\!,\alpha} } = 1
\end{equation}
and the infimum is achieved on functions of the form $h^\alpha_z(w)=e^{\alpha\langle w, z\rangle}$ for any $z\in\C^n$. Secondly, for any $z\in\C^n$, and any $g\in\mathcal{H}^n_{p,\alpha}$,
\begin{equation} \label{e.ptwise.bd2} 
|g(z)| \le e^{\frac{\alpha}{2}|z|^2}\|g\|_{p,\alpha}
\end{equation} 
and the above equation is sharp.

As we stated in the introduction, \eqref{e.ptwise.bd2} is well-known. In fact one can use it to prove the first part of the Theorem \ref{t.1}.  To see how, we first record a fact (first proved in \cite{JPR}) that will be useful in the following arguments as well.  The projection $P^n_\alpha\colon L^2(\gamma^n_\alpha)\to \mathcal{H}^n_{p,\alpha}$ is given by the integral operator
\begin{equation} \label{e.Pn.int} (P^n_\alpha g)(z) = \int_{\C^n} e^{\alpha\langle z,w\rangle}g(w) \,\gamma^n_\alpha(dw) = \langle g,h_z^\alpha\rangle_\alpha, \end{equation}
where for any $z\in\mathbb{C}^n$ we define the function $h_z^\alpha$ as
\begin{equation}h_z^\alpha(w) = e^{\alpha\langle w,z\rangle}.\nonumber\end{equation}
Since polynomials are dense in $L^2(\gamma^n_\alpha)$, we may extend this integral operator to act densely on any space in which polynomials are dense.  The first main theorem of \cite{GrycKemp2011} shows that $P^n_\alpha$ is, in fact, bounded on $L^p(\gamma^n_{\alpha p/2})$, with image in $\mathcal{H}^n_{p,\alpha}$.  Now, any holomorphic polynomial $g$ over $\C^n$ is in $\mathcal{H}^n_{2,\alpha}$ and so $P^n_\alpha g=g$; since holomorphic polynomials are dense in $\mathcal{H}^n_{p,\alpha}$, it therefore follows that
\begin{equation} \label{e.Pn.rk} (P^n_\alpha g)(z)=\langle g,h_z^\alpha\rangle_\alpha = g(z), \qquad \text{for all} \qquad z\in \C^n, \; g\in \mathcal{H}^n_{p,\alpha}. \end{equation}
\begin{remark} Since $\mathcal{H}^n_{p,\alpha}\subset L^p(\gamma^n_{\alpha p/2})$, it might seem more natural to expect that the reproducing formula for functions in $\mathcal{H}^n_{p,\alpha}$ should involve the reproducing kernel $h_z^{\alpha p/2}(w) = e^{\frac{\alpha p}{2}\langle z,w\rangle}$.  This would be true if the inner product used was $\langle\cdot,\cdot\rangle_{\alpha p/2}$; it is a remarkable and useful fact that, using the fixed inner product $\langle\cdot,\cdot\rangle_\alpha$ for all $\mathcal{H}^n_{p,\alpha}$ spaces gives a consistent reproducing kernel for all of them.  \end{remark}

Now we can state the following lemma:
\begin{lemma}\label{lem.t.1.1}
Let $1<p<\infty$. If \eqref{e.ptwise.bd2} holds for all $g\in\mathcal{H}^n_{p,\alpha}$ and $z\in\C^n$, then \eqref{t.1.1} holds.
\end{lemma}
\begin{proof}
Assume \eqref{e.ptwise.bd2} holds for all $g\in\mathcal{H}^n_{p,\alpha}$.  Note that by \eqref{e.norm.equiv} we know that 
\begin{equation}\label{e.2.0}
\inf_{h\in\mathcal{H}^n_{p'\!,\alpha}\setminus\{0\}} \frac{  \| \langle \cdot,h\rangle_\alpha\|_{p,\alpha}^\ast}{ \|h\|_{p'\!,\alpha} } \geq 1
\end{equation}
Assume that \eqref{e.ptwise.bd2} holds for functions $g$ in $\mathcal{H}^n_{p,\alpha}$.  For any fixed $z\in\C^n$, the functional $\langle\cdot, h^\alpha_z\rangle_\alpha$ is pointwise evaluation at $z$, cf.\ \eqref{e.Pn.rk}.  Thus, by \eqref{e.ptwise.bd2}, we know that
\begin{equation}
\|\langle\cdot, h_z^\alpha\rangle_\alpha\|_{p,\alpha}^\ast \le e^{\frac{\alpha}{2}|z|^2}\nonumber
\end{equation}
However, one can easily calculate using Lemma \ref{l.Gauss.int} that 
\begin{equation}
\|h_z^\alpha\|_{p',\alpha} = e^{\frac{\alpha}{2}|z|^2}\nonumber,
\end{equation}
proving that $\|\langle\cdot, h_z^\alpha\rangle\|_{p,\alpha}^\ast\le\|h_z^\alpha\|_{p',\alpha}$.  Thus, by \eqref{e.2.0} we have proven \eqref{t.1.1} and shown that this infimum is achieved at each $h_z^\alpha$, as desired.
\end{proof}

Note that Lemma \ref{lem.t.1.1} actually proves Theorem \ref{t.1} since \eqref{e.ptwise.bd2} is known to hold for all $g\in\mathcal{H}^n_{p,\alpha}$ and $z\in\C^n$; it is often referred to as Bargmann's inequality.  However, we have an alternate proof of Theorem \ref{t.1} that proves the result independently of the a priori truth of \eqref{e.ptwise.bd2}.  That is, without assuming \eqref{e.ptwise.bd2} is true, we can prove both \eqref{t.1.1} and \eqref{e.ptwise.bd2}.  This proof is based on the following lemma: 
\begin{lemma}\label{l.2.1}
Let $n\in\N$, $\alpha>0$, and $1<p<\infty$ with $\frac{1}{p}+\frac{1}{p'}=1$.  Let $h\in\mathcal{H}_{p',\alpha}^n$.  Then
	\begin{equation*}
	\|\langle\cdot,h\rangle_{\alpha}\|_{p,\alpha}^\ast=C_p^n\inf_{f\in P_\alpha^{-1} h} \|f\|_{p',\alpha}.
	\end{equation*}
	Furthermore, there exists a function $\tilde{f}\in P_\alpha^{-1} h$ where the infimum is achieved.
\end{lemma}
The rest of this section is devoted to proving the Lemma \ref{l.2.1} and using it to prove Theorem \ref{t.1}.

\subsection{A Relationship Between the Norm, the Dual Norm, and the Projection $P^n_\alpha$}
Before we prove Lemma \ref{l.2.1}, we need some preliminary results.  While we refer to the mapping $P^n_\alpha:L^p(\gamma^n_{\alpha p/2})\to\mathcal{H}_{p,\alpha}^n$ as a ``projection,'' it is of course not a true orthogonal projection for $p\neq 2$ as in this case $\mathcal{H}_{p,\alpha}^n$ is not a Hilbert space.  However, it acts like a projection in the following ways (as proven in \cite{JPR}):  $P^n_\alpha$ is the identity on elements in $\mathcal{H}_{p,\alpha}^n$ (this was actually shown in \eqref{e.Pn.rk}) and $P^n_\alpha$ is ``self-adjoint'' in the following sense: 
\begin{equation}
\langle P^n_\alpha g, h\rangle_\alpha = \langle g, P^n_\alpha h\rangle_\alpha\mbox{ for all $(g,h)\in L^p(\gamma^n_\alpha)\times L^{p'}(\gamma^n_\alpha)$.}\label{sa1}
\end{equation}
As we alluded to in the Introduction, to prove a statement about the spaces $\mathcal{H}_{p,\alpha}^n$ it can be useful to prove an analogous statement in a corresponding Lebesgue measure setting.  Indeed, the differing measures of $\gamma^n_{\alpha p/2}$ and $\gamma^n_{\alpha p'/2}$ preclude us from using some basic results of duality in $L^p$ spaces.  To remove this complication, we define a mapping $\mathfrak{g}^n_{p,\alpha}:L^p(\gamma^n_{\alpha p/2})\to L^p(\C^n, \lambda^n)$ as
\begin{equation}
(\mathfrak{g}^n_{p,\alpha} f)(z) = \left(\frac{p\alpha}{2\pi}\right)^{n/p} e^{-\frac{\alpha}{2}|z|^2}f(z). \nonumber
\end{equation}
It is easy to check that $\mathfrak{g}^n_{p,\alpha}$ is an isometric isomorphism.  Furthermore, define the set $\mathcal{S}^n_{p,\alpha}$ as the image of $\mathcal{H}_{p,\alpha}^n$ under $\mathfrak{g}^n_{p,\alpha}$ above.  That is,
\begin{equation}
\mathcal{S}^n_{p,\alpha}=\{F: \|F\|_{p,\lambda}<\infty,\, z\mapsto F(z)e^{\frac{\alpha}{2}|z|^2}\mbox{ is holomorphic}\}\nonumber
\end{equation}
where $\|F\|_{p,\lambda}$ is the $L^p(\C^n,\lambda^n)$ norm of $F$. The space $\mathcal{S}^n_{p,\alpha}$ is the set of so-called ``holomorphic sections'' mentioned in the introduction.  Using the isomorphism $\mathfrak{g}^n_{p,\alpha}$ one can see that $(\mathcal{S}^n_{p,\alpha})^\ast=\mathcal{S}^n_{p',\alpha}$ as identified using the usual Lebesgue integral pairing $(G,H)_\lambda=\int_{\C^n} G\overline{H}\, d\lambda^n$.

Define a new operator $Q^n_\alpha:L^p(\C^n, \lambda^n)\to L^p(\C^n,\lambda^n)$ as 
	\begin{equation}
	Q^n_\alpha=\mathfrak{g}^n_{p,\alpha} P^n_\alpha \left(\mathfrak{g}^n_{p,\alpha}\right)^{-1}.\nonumber
	\end{equation}
	Note that $Q^n_\alpha$ does not actually depend on $p$.  Indeed, $\mathfrak{g}^n_{p,\alpha}$ only depends on $p$ through multiplication by $p$-dependent constant.  From this fact, it is easy to see that  
	\begin{equation*}
	Q^n_\alpha =\mathfrak{g}_{\alpha,2} P^n_\alpha \mathfrak{g}_{\alpha,2}^{-1},
	\end{equation*}
	justifying the notation.  By definition, the following diagram commutes:
	$$\xymatrix{
L^p(\gamma_{\alpha p /2}) \ar[r]^{\mathfrak{g}^n_{p,\alpha}} \ar[d]_{P^n_\alpha} & L^p(\C^n,\lambda^n) \ar[d]^{Q^n_\alpha} \\
\mathcal{H}_{p,\alpha}^n \ar[r]_{\mathfrak{g}^n_{p,\alpha}} & \mathcal{S}^n_{p,\alpha}}$$
	Denote by  $\|Q^n_\alpha\|_{p\to p}$ the norm of $Q^n_\alpha$ as an operator on $L^p(\C^n,\lambda^n)$ and $\|P^n_\alpha\|_{p\to p}$ the norm of $P^n_\alpha$ as an operator on $L^p(\gamma^n_{\alpha p/2})$. Since $\mathfrak{g}^n_{p,\alpha}$ and its inverse are isometric, it is not difficult to show that $P^n_\alpha$ and $Q^n_\alpha$ share many similar properties.  Specifically,
	\begin{enumerate}
	\item $\|Q^n_\alpha\|_{p\to p}=\|P^n_\alpha\|_{p\to p}$,
	\item $Q^n_\alpha$ is the identity on $\mathcal{S}^n_{p,\alpha}$ and maps onto $\mathcal{S}^n_{p,\alpha}$ for $1<p<\infty$, and
	\item $Q^n_\alpha$ is ``self-adjoint'' in the sense of \eqref{sa1} in the pairing $(G,H)_\lambda=\int_{\C^n} G\overline{H}\, d\lambda^n$.
	\end{enumerate}	
To precisely state the third fact above, we write
	\begin{equation}
	(Q^n_\alpha G, H)_\lambda = (G, Q^n_\alpha H)_\lambda\qquad \text{for all}\qquad (G,H)\in L^p(\C^n,\lambda^n)\times L^{p'}(\C^n,\lambda^n). \label{sa2}
	\end{equation}
	We can now state and prove a result analogous to Lemma \ref{l.2.1} for the space of holomorphic sections.  Below $\|\cdot\|_{(\mathcal{S}^n_{p,\alpha})^*}$ denotes the dual norm of $\mathcal{S}^n_{p,\alpha}$.
	\begin{lemma}\label{l.2.2}
	Let $H\in\mathcal{S}^n_{p',\alpha}$.  Then
	\begin{equation*}
\|(\cdot,H)_\lambda\|_{(\mathcal{S}^n_{p,\alpha})^*}=\inf_{F\in (Q^n_\alpha)^{-1} H} \|F\|_{p',\lambda}.
	\end{equation*}
	Furthermore, there exists some $\tilde{F}\in (Q^n_\alpha)^{-1} H \subseteq L^{p'}(\C^n, \lambda^n)$ such that
	\begin{equation*}
	\inf_{F\in (Q^n_\alpha)^{-1} H} \|F\|_{p',\lambda}=\|\tilde{F}\|_{p',\lambda}.
	\end{equation*}
	\end{lemma}
	\begin{proof}[Proof of Lemma \ref{l.2.2}]
	Let $H\in\mathcal{S}^n_{p',\alpha}$ be arbitrary.  We prove the first equation of the lemma by showing that
	\begin{eqnarray}
	\|(\cdot,H)_\lambda\|_{(\mathcal{S}^n_{p,\alpha})^*}&\leq&\inf_{F\in (Q^n_\alpha)^{-1} H} \|F\|_{p', \lambda}, \quad \text{and}\label{inflemma3}\\
	\|(\cdot,H)_\lambda\|_{(\mathcal{S}^n_{p,\alpha})^*}&\geq&\inf_{F\in (Q^n_\alpha)^{-1} H} \|F\|_{p', \lambda}\label{inflemma4}.
	\end{eqnarray}
	
	We first prove \eqref{inflemma3}. Let $F\in (Q^n_\alpha)^{-1} H$ be arbitrary.  Then
	\begin{eqnarray*}
\|(\cdot,H)_\lambda\|_{(\mathcal{S}^n_{p,\alpha})^*}&=&\sup_{G\in\mathcal{S}^n_{p,\alpha}}\frac{|(G,H)_\lambda|}{\|G\|_{p,\lambda}}=\sup_{G\in\mathcal{S}^n_{p,\alpha}}\frac{|(G,Q^n_\alpha F)_\lambda|}{\|G\|_{p,\lambda}}=\sup_{G\in\mathcal{S}^n_{p,\alpha}}\frac{|(Q^n_\alpha G, F)_\lambda|}{\|G\|_{p,\lambda}}
= \sup_{G\in\mathcal{S}^n_{p,\alpha}}\frac{|(G, F)_\lambda|}{\|G\|_{p,\lambda}}\\
&\leq& \sup_{G\in L^p(\C^n,\lambda^n)}\frac{|(G, F)_\lambda|}{\|G\|_{p,\lambda}}=\|F\|_{p', \lambda}.
\end{eqnarray*}
	Since $F\in (Q^n_\alpha)^{-1} H$ was arbitrary, we have proven \eqref{inflemma3}.

	To prove \eqref{inflemma4}, define the linear functional 		$\Lambda:\mathcal{S}^n_{p,\alpha}\to\C$ as
	$$\Lambda(G)=(G,H)_\lambda.$$
	Note that $\|(\cdot,H)_\lambda\|_{(\mathcal{S}^n_{p,\alpha})^*}=\|\Lambda\|$, so that $\Lambda$ is bounded.  By the Hahn-Banach theorem there is a linear functional $\tilde{\Lambda}:L^p(\C^n,\lambda^n)\to\C$ that extends $\Lambda$ without increasing its norm.  As $\tilde{\Lambda}\in (L^p(\C^n,\lambda))^*$, there exists a function $\tilde{F}\in L^{p'}(\C^n,\lambda^n)$ such that 
$$\tilde{\Lambda}(G)=(G,\tilde{F})_\lambda.$$
First note that for any $G\in L^{p}(\mathbb{C}^n,\lambda^n)$ we have 
$$(G,Q^n_\alpha \tilde{F})_\lambda=(Q^n_\alpha G,\tilde{F})_\lambda=\tilde{\Lambda}(Q^n_\alpha G)=\Lambda(Q^n_\alpha G)=(Q^n_\alpha G, H)_\lambda=(G,H)_\lambda,$$
proving $\tilde{F}\in (Q^n_\alpha)^{-1} H$. Then
\begin{equation*}
\|(\cdot,H)_\lambda\|_{(\mathcal{S}^n_{p,\alpha})^*}=\|\tilde{\Lambda}\|=\sup_{G\in L^p(\mathbb{C}^n,\lambda^n)}\frac{|(G, \tilde{F})_\lambda|}{\|G\|_{p,\lambda}}=\|\tilde{F}\|_{p', \lambda}\geq \inf_{G\in (Q^n_\alpha)^{-1} H} \|G\|_{p',\lambda},
\end{equation*}
proving \eqref{inflemma4}.  Combining \eqref{inflemma3} and the preceding inequality, we see that $\|\tilde{F}\|_{p',\lambda}=\inf_{G\in (Q^n_\alpha)^{-1} H} \|G\|_{p', \lambda}$, completing the lemma.
	\end{proof}
	We can now provide a proof for Lemma \ref{l.2.1}:
	
	\begin{proof}[Proof of Lemma \ref{l.2.1}]
	For $g\in L^p(\gamma^n_{\alpha p/2})$ and $h\in L^{p'}(\gamma^n_{\alpha p'/2})$, a straightforward calculation reveals that
	\begin{equation}
\langle g,h\rangle_{\alpha}=C_p^n\cdot (\mathfrak{g}^n_{p,\alpha}g,\mathfrak{g}^n_{p',\alpha}h)_\lambda\label{inflemma5}.
	\end{equation}
	Note that the constant $C_p^n$ pops up above since we are combining two different isometries:  $\mathfrak{g}^n_{p,\alpha}$ and $\mathfrak{g}^n_{p',\alpha}$. A straightforward combination of \eqref{inflemma5} and Lemma \ref{l.2.2} completes the proof.
	\end{proof}

	\begin{remark}\label{r.2.1}
	Before moving on to a proof of Theorem \ref{t.1}, we note here that we can use Lemma \ref{l.2.1} to rederive \eqref{e.norm.equiv}.  That is, the inequality 
	\begin{equation}  
	\|h\|_{p'\!,\alpha} \le \| \langle \cdot,h\rangle_\alpha\|_{p,\alpha}^\ast \le C_p^n  \|h\|_{p'\!,\alpha}\nonumber. 
	\end{equation} 
	Let $h\in\mathcal{H}_{p,\alpha}^n$ be arbitrary. For the first inequality, note that for any $f\in (P^n_\alpha)^{-1} h$ 
	\begin{equation}\label{i.2.1}
	\frac{\|h\|_{p'\!,\alpha}}{\|P^n_\alpha\|_{p'\to p'}}\leq \|f\|_{p'\!,\alpha}.
	\end{equation}
	Thus, 
$$\frac{\|h\|_{p'\!,\alpha}}{\|P^n_\alpha\|_{p'\to p'}}\leq \inf_{f\in P_\alpha^{-1} h}\|f\|_{p'\!,\alpha}.$$
	Also, $h\in P_\alpha^{-1} h$, so that 
$$\inf_{f\in P_\alpha^{-1} h} \|f\|_{p'\!, \alpha}\leq \|h\|_{p'\!,\alpha}.$$
	Putting these inequalities together gives us
	\begin{equation}
\frac{\|h\|_{p'\!,\alpha}}{\|P^n_\alpha\|_{p'\to p'}}\leq\inf_{f\in P_\alpha^{-1} h} \|f\|_{p'\!, 	\alpha}\leq \|h\|_{p'\!,\alpha}.\nonumber
	\end{equation}
	Using Lemma \ref{l.2.1} and the fact that $\|P^n_\alpha\|_{p'\to p'}=C_p^n$ (from \cite{GrycKemp2011}) in the above equation gives us
 	\begin{equation}
\frac{\|h\|_{p'\!,\alpha}}{C_p^n}\leq \frac{\|\langle\cdot,h\rangle_\alpha\|_{p,\alpha}^\ast}{C_p^n}\leq \|h\|_{p'\!,\alpha}.\nonumber
	\end{equation}
	Multiplying the above by $C_p^n$ gives us a proof of \eqref{e.norm.equiv}. 
	\end{remark}	
	\subsection{Proof of Theorem \ref{t.1} Using Lemma \ref{l.2.1}}
	We are now ready to prove Theorem \ref{t.1}.
	We first prove
	\begin{equation}\label{inf1}
	\inf_{h\in\mathcal{H}^n_{p'\!,\alpha}\setminus\{0\}} \frac{  \| \langle \cdot,h\rangle_\alpha\|_{p,\alpha}^\ast}{ \|h\|_{p'\!,\alpha} } = 1
	\end{equation}
	and that this infimum is achieved.  Note that by the proof of \eqref{e.norm.equiv} in Remark \ref{r.2.1}, to prove equality in \eqref{inf1} it suffices to show that there exists some $h\in\mathcal{H}_{p,\alpha}^n$ and $f\in (P^n_\alpha)^{-1} h$ such that there is equality in \eqref{i.2.1}. That is
	\begin{equation}\label{i.2.2}
	\frac{\|h\|_{p'\!,\alpha}}{C_p^n}=\|f\|_{p'\!,\alpha}.
	\end{equation}
	Let $f(z)=\left(\frac{p\alpha}{2\pi}\right)^{p/n}e^{-\frac{\alpha}{2}|z|^2}$; then $P^n_\alpha f \equiv 1 \equiv h^\alpha_0$ (the $z=0$ case of the function $h^\alpha_z(w)=e^{\alpha\langle w,z\rangle}$).  A straightforward computation shows that $h=h_0^\alpha$ and $f$ satisfy \eqref{i.2.2}.  This proves \eqref{inf1}.
	
	Now, using \eqref{e.Pn.rk}, we have $\langle g,h_0^\alpha\rangle_\alpha = g(0)$.
	We just showed that $\| \langle \cdot,h_0\rangle_\alpha\|_{p,\alpha}^\ast=\|h_0\|_{p'\!,\alpha}$, which means that the following inequality is sharp:
	\begin{equation}\label{e.pointwise0}
	|g(0)|\leq \|g\|_{p,\alpha} \qquad \text{for all} \qquad g\in \mathcal{H}^n_{p,\alpha}.
	\end{equation}
	Let $z\in\C^n$ be arbitrary.  Let $g\in \mathcal{H}^n_{p,\alpha}$ be arbitrary.  Define a new function $g_z(w)=g(z+w)e^{-\alpha\langle w+z,z\rangle}$.  Note that $g_z$ is holomorphic and
	\begin{eqnarray*}
	\|g_z\|^p_{p,\alpha}&=&\left(\frac{\alpha p}{2\pi}\right)^n\int_{\C^n} 	|g(z+w)e^{-\alpha\langle w+z,z\rangle}|^p e^{-\alpha p|w|^2/2}\,\lambda^n(dw)\\
	&=&\left(\frac{\alpha p}{2\pi}\right)^n \int_{\C^n} |g(y)e^{-\alpha\langle y,z\rangle}|^p 	e^{-\alpha p|y-z|^2/2}\,\lambda^n(dy)\\
	&=&e^{-\alpha p|z|^2/2}\int_{\C^n} |g(y)|^p \left(\frac{\alpha p}{2\pi}\right)^n  e^{-\alpha 	p|y|^2/2}\,\lambda^n(dy)
	=e^{-\alpha p|z|^2/2}\|g\|^p_{p,\alpha}<\infty,
	\end{eqnarray*}
	proving that $g_z\in\mathcal{H}^n_{p,\alpha}$.  Applying \eqref{e.pointwise0} to $g_z$ yields the inequality
	\begin{equation}\label{e.pointwise1}
	|g(z)|\leq e^{\alpha |z|^2/2}\|g\|_{p,\alpha}\mbox{ for all $g\in \mathcal{H}^n_{p,\alpha}$.}
	\end{equation}
	A straightforward calculation shows that the inequality \eqref{e.pointwise1} is an equality when $g=h_z^\alpha$, proving the inequality sharp.  The sharpness of \eqref{e.pointwise1} proves that $\|\langle\cdot, h_z^\alpha \rangle_\alpha\|_{p,\alpha}^\ast=e^{\alpha |z|^2/2}=\|h_z^\alpha\|_{p'\!,\alpha}$, proving the infimum \eqref{inf1} is achieved at each $h_z^\alpha$ and completing the proof of Theorem \ref{t.1}.

\section{The Strictness of H\"{o}lder's Inequality and a Lower Bound for the Sharp Upper Constant}\label{s.3}
As Theorem \ref{t.1} is proven, we know that the left-hand inequality of \eqref{e.norm.equiv} is sharp.  For the remainder of the paper, we will consider the right-hand inequality, that is
\begin{equation}\label{e.norm.equiv.right}
\| \langle \cdot,h\rangle_\alpha\|_{p,\alpha}^\ast \le C_p^n  \|h\|_{p'\!,\alpha}.
\end{equation}
As we stated in the introduction, we do not know whether \eqref{e.norm.equiv.right} is sharp, but Theorems \ref{t.2}, \ref{t.3}, \ref{t.4.new} and \ref{t.5} suggest that it is not sharp.  We presently prove Theorems \ref{t.2} and \ref{t.3}.

\subsection{The Proof of Theorem \ref{t.2}}

Here will prove that H\"{o}lder's inequality is not sharp in the Segal-Bargmann spaces.  Let $1<p<\infty$, $p\neq 2$, $g\in\mathcal{H}^n_{p,\alpha}$, and $h\in\mathcal{H}^n_{p',\alpha}$, neither identically $0$.  We will proceed by contradiction.  That is, suppose that $g$ and $h$ give equality in H\"{o}lder's inequality (modified by the constant $C_p^n$ to account for the scaling of the spaces $\mathcal{H}^n_{p,\alpha}$).  Thus,
\begin{eqnarray}
|\langle g, h\rangle_\alpha|
&\leq&  \int_{\C^n} |g(z)\overline{h(z)}|\,\gamma^n_\alpha(dz)\label{holder1}\\
&=& \left(\frac{\alpha}{\pi}\right)^n\int_{\C^n} |g(z)e^{-\alpha |z|^2/2}||h(z)e^{-\alpha |z|^2/2}|\,\lambda^n(dz)\nonumber\\
&\leq& \left(\frac{\alpha}{\pi}\right)^n \left(\int_{\C^n}|g(z)e^{-\alpha |z|^2/2}|^p\,\lambda^n(dz)\right)^{1/p}\left(\int_{\C^n}|h(z)e^{-\alpha |z|^2/2}|^{p'}\,\lambda^n(dz)\right)^{1/p'}\label{holder2}\\
&=& \left(\frac{\alpha}{\pi}\right)^n\left(\frac{2\pi}{p\alpha}\right)^{n/p}\left(\frac{2\pi}{p'\alpha}\right)^{n/p'}\|g\|_{L^{p}(\gamma_{\alpha p/2})}\|h\|_{L^{p'}(\gamma_{\alpha p'/2})}\nonumber\\
&=& C_p^n\,\|g\|_{L^{p}(\gamma_{\alpha p/2})}\|h\|_{L^{p'}(\gamma_{\alpha p'/2})}=|\langle g, h\rangle_\alpha|\nonumber,
\end{eqnarray}
proving that both \eqref{holder1} and \eqref{holder2} are actually equalities.  For equality in \eqref{holder2}, we must have
$$|g(z)e^{-\alpha |z|^2/2}|^p=\beta^p|h(z)e^{-\alpha |z|^2/2}|^{p'}$$
for some $\beta>0$.  Rearranging the above gives us
\begin{equation}\label{holder3}
|g(z)|=\beta|h(z)|^{{p'}/p}e^{-\frac{\alpha(p'-p)}{2p}|z|^2}.
\end{equation}
For \eqref{holder1} to be an equality, we must have
\begin{equation}\label{holder4}
g(z)\overline{h(z)}=e^{i\theta_0} f(z),
\end{equation}
where $\theta_0\in [0,2\pi]$ and $f$ is a nonnegative real-valued function.  By replacing $g(z)$ with $\beta^{-1} e^{-i\theta_0}g(z)$, we preserve holomorphicity and the finiteness of the $\|\cdot\|_{p,\alpha}$-norm.  Thus, without loss of generality, we may assume that $e^{i\theta_0}=\beta=1$, and replace Equations \eqref{holder3} and \eqref{holder4} with 
\begin{equation}\label{holder5}
|g(z)|=|h(z)|^{{p'}/p}e^{-\frac{\alpha(p'-p)}{2p}|z|^2},
\end{equation}
and
\begin{equation}\label{holder6}
g(z)\overline{h(z)}=f(z),\mbox{ where $f$ is non-negative real-valued.}
\end{equation}

Since $g,h$ are holomorphic and not identically $0$, they are each non-zero on an open dense subset of $\C^n$; thus, there is an open set $U$ where neither $g$ nor $h$ vanishes.  Then $g/h$ is holomorphic on $U$, and
\[ \frac{g(z)}{h(z)} = \frac{g(z)\overline{h(z)}}{|h(z)|^2} = \frac{f(z)}{|h(z)|^2} >0 \quad \text{for} \quad z\in U.\]
Thus, $g/h$ is a positive holomorphic function, and so it is equal to a positive constant $c$ on $U$.   Equation\ \eqref{holder5} then shows that
\[ c|h(z)| = |h(z)|^{p'/p}e^{-\frac{\alpha(p'-p)}{2p}|z|^2}. \]
Solving for $|h(z)|$ above and raising each side to the $\frac{p}{p-p'}$ power (which is possible as $p\neq 2$ and thus $p\neq p'$) gives
\[ |h(z)| = c_1 e^{\frac{\alpha}{2}|z|^2}, \qquad c_1 = c^{\frac{1}{p'/p-1}}. \]
Fix any point $z=(z_1,\ldots,z_n)\in U$; then there is some disk $D\subset\C$ such that $\{(\zeta,z_2,\ldots,z_n)\colon \zeta\in D\}\subset U$.  Thus the function $h_1(\zeta)=h(\zeta,z_2,\ldots,z_n)$ is holomorphic and non-vanishing on $D$, and we have
\[ |h_1(\zeta)| = c_1e^{\frac{\alpha}{2}(|\zeta|^2 + |z_2|^2+\cdots + |z_n|^2)}. \]
The function $h_2(\zeta) = c_1^{-1}e^{-\frac{\alpha}{2}(|z_2|^2+\cdots+|z_n|^2)}h_1(\zeta)$ is therefore holomorphic and non-vanishing on $D$, and $|h_2(\zeta)| =e^{\frac{\alpha}{2}|\zeta|^2}$.  It follows that $h_1$ has a holomorphic logarithm $\ell$ on $D$, so
\[ e^{\frac{\alpha}{2}|\zeta|^2} = |h_2(\zeta)| = |e^{\ell(\zeta)}| = e^{\Re\ell(\zeta)}, \quad z\in D. \]
As $\exp$ is one-to-one on $\R$, it follows that $\Re\ell(\zeta) = \frac{\alpha}{2}|\zeta|^2$ for $\zeta\in D$.  This is impossible, since $\ell$ is holomorphic, but $\zeta\mapsto\frac{\alpha}{2}|\zeta|^2$ is not harmonic. This concludes the proof.

\subsection{The Proof of Theorem \ref{t.3}\label{s.proof of t.3}}

As in the Introduction, define $\mathcal{R}_{p,\alpha}(g,h)$ as 
\[ \mathcal{R}_{p,\alpha}(g,h) = \frac{ |\langle g,h\rangle_\alpha | } { \|g\|_{p,\alpha} \|h\|_{p'\!,\alpha} }. \]
Note that the sharp constant for \eqref{e.norm.equiv.right} is equal to $\sup_{g\in \mathcal{H}^n_{p,\alpha}\setminus\{0\}}\; \sup_{h\in \mathcal{H}^n_{p'\!,\alpha}\setminus\{0\}} \mathcal{R}_{p,\alpha}(g,h)$, hence our interest in this ratio.  Theorem \ref{t.3} concerns bounds on this ratio; namely that  \eqref{e.sharp?}, reproduced below, holds:
\begin{equation}C_p^{n/2} \le \sup_{g\in \mathcal{H}^n_{p,\alpha}\setminus\{0\}}\; \sup_{h\in \mathcal{H}^n_{p'\!,\alpha}\setminus\{0\}} \mathcal{R}_{p,\alpha}(g,h) \le C_p^n. \nonumber\end{equation}

There are many ways to prove the right-hand side of \eqref{e.sharp?}. In particular, we can rewrite Theorem \ref{t.2} in terms of $\mathcal{R}_{p,\alpha}(g,h)$ to say that for any $g\in\mathcal{H}^n_{p,\alpha}$ and $h\in\mathcal{H}^n_{p',\alpha}$ we have
\begin{equation}
\mathcal{R}_{p,\alpha}(g,h)< C_p^n.\nonumber
\end{equation}
By the above, we have
\begin{equation}
\sup_{g\in \mathcal{H}^n_{p,\alpha}\setminus\{0\}}\; \sup_{h\in \mathcal{H}^n_{p'\!,\alpha}\setminus\{0\}} \mathcal{R}_{p,\alpha}(g,h)\leq C_p^n.
\end{equation}
Thus, we need only prove the left-hand inequality of \eqref{e.sharp?}.  To that end, we will consider the case where $g$ and $h$ are monomials.  Note that (by the rotational invariance of $\gamma^n_\alpha$) distinct monomials are orthogonal, so we will consider only $g=h$.  For $k_1,\ldots,k_n\in\N$, define the $g_{k_1,k_2,\ldots,k_n}(z)\equiv z_1^{k_1}z_2^{k_2}\ldots z_n^{k_n}$.  Note that
\begin{eqnarray*}
\|g_{k_1,k_2,\ldots,k_n}\|_{p, \alpha}^p&=& \left(\frac{\alpha p}{2\pi}\right)^n\int_{\mathbb{C}^n} |z_1|^{k_1p}|z_2|^{k_2p}\ldots |z_n|^{k_np} e^{-(\alpha p/2) (|z_1|^2+|z_2|^2+\ldots+|z_n|^2)} \lambda^n(dz)\\
&=&\left(\frac{\alpha p}{2\pi}\right)^n\prod_{j=1}^n \int_{\mathbb{C}}|z_j|^{pk_j}e^{-(\alpha p/2)|z_j|^2}\,\lambda(dz_j)\\
&=&\prod_{j=1}^n \|g_{k_j}\|_{p, \alpha}^p
\end{eqnarray*}
where $g_k:\C\to\C$ is given by $g_k(z)=z^k$.  Note, then, that
\begin{equation} \label{e.tensor.g} \mathcal{R}_{p,\alpha}(g_{k_1,\ldots,k_n},g_{k_1,\ldots,k_n}) = \prod_{j=1}^n \mathcal{R}_{p,\alpha}(g_{k_j},g_{k_j}). \end{equation}
Hence, to prove the left-hand side of \eqref{e.sharp?}, it suffices to show that
\begin{equation} \label{e.sharp?2} \sup_{k\in\N} \mathcal{R}_{p,\alpha}(g_k,g_k)= C_p^{1/2}. \end{equation}
As usual, denote the Gamma function $\Gamma(z)$ as
$$\Gamma(z)=\int_0^\infty t^{z-1}e^{-t}dt, \qquad \Re (z)>0.$$
Then, using polar coordinates, we have 
\begin{align*}
\|g_k\|_{p, \alpha}^p &= \left(\frac{\alpha p}{2\pi}\right)\int_{\mathbb{C}} |z|^{kp} e^{-(\alpha p/2) |z|^2}\,\lambda(dz)
= \left(\frac{\alpha p}{2\pi}\right)\int_0^\infty r\int_{S^{1}} r^{kp}e^{-(\alpha p/2) r^2}dSdr \\
&= (\alpha p) \int_0^\infty (r^2)^{kp/2}re^{-(\alpha p/2) r^2}dr.
\end{align*}
Using the substitution $u=\frac{\alpha p}{2}r^2$ yields
\begin{eqnarray*}
\|g_k\|_{p, \alpha}^p&=& \int_0^\infty (r^2)^{kp/2}e^{-(\alpha p/2) r^2}(\alpha p)rdr\\
&=& \left(\frac{2}{\alpha p}\right)^{kp/2} \int_0^\infty u^{kp/2}e^{-u}du = \left(\frac{2}{\alpha p}\right)^{kp/2}\Gamma(kp/2+1).
\end{eqnarray*}
Thus, we have
\begin{align} \nonumber
\mathcal{R}_{p,\alpha}(g_k,g_k)=\frac{|\langle g_k,g_k\rangle_\alpha|}{\|g_k\|_{p, \alpha}\|g_k\|_{p'\!, \alpha}}&=\frac{\|g_k\|^2_{2,\alpha}}{\|g_k\|_{p, \alpha}\|g_k\|_{p'\!, \alpha}} \\
&=\left(\frac{pp'}{4}\right)^{k/2}\frac{\Gamma(k+1)}{\Gamma(kp/2+1)^{1/p}\Gamma(kp'/2+1)^{1/p'}}. \label{e.Rp1}
\end{align}
Using the Gamma function relation $\Gamma(z+1)=z\Gamma(z)$, it is convenient to express this ratio as
\begin{align} \nonumber \frac{\Gamma(k+1)}{\Gamma(kp/2+1)^{1/p}\Gamma(kp'/2+1)^{1/p'}} &= \frac{k}{(kp/2)^{1/p}(kp'/2)^{1/p'}}\cdot\frac{\Gamma(k)}{\Gamma(kp/2)^{1/p}\Gamma(kp'/2)^{1/p'}} \\
&= C_p\cdot \frac{\Gamma(k)}{\Gamma(kp/2)^{1/p}\Gamma(kp'/2)^{1/p'}}. \label{e.Rp2}
\end{align}

To properly analyze this expression, we will use a precise form of Stirling's approximation for the Gamma function: for any $z\in\C$ with $\Re (z)>0$,
\begin{equation} \label{e.newS} S(z)\equiv \ln\left(\sqrt{\frac{z}{2\pi}}\left(\frac{e}{z}\right)^z\Gamma(z)\right) = \int_0^\infty\frac{2\arctan(t/z)}{e^{2\pi t}-1}\,dt. \end{equation}
See, for example, \cite[(6.1.50)]{AbramowitzStegun}. Thus, we can express the Gamma function precisely as
\begin{equation} \label{e.Gamma.S} \Gamma(z) = \sqrt{2\pi}\,z^{z-\frac12}\,e^{S(z)-z}. \end{equation}
With this in hand, together with \eqref{e.Rp1} and \eqref{e.Rp2}, we have the following expression for $\mathcal{R}_{p,\alpha}(g_k,g_k)$.
\begin{align*} \mathcal{R}_{p,\alpha}(g_k,g_k) &= C_p\cdot \left(\frac{pp'}{4}\right)^{k/2}\frac{\Gamma(k)}{\Gamma(kp/2)^{1/p}\Gamma(kp'/2)^{1/p'}} \\
&= C_p\cdot \left(\frac{pp'}{4}\right)^{k/2}\frac{\sqrt{2\pi}k^{k-\frac12}e^{S(k)-k}}{\left(\sqrt{2\pi}(kp/2)^{kp/2-\frac12} e^{S(kp/2)-kp/2}\right)^{1/p}\left(\sqrt{2\pi}(kp'/2)^{kp'/2-\frac12} e^{S(kp'/2)-kp'/2}\right)^{1/p'}} \\
&= C_p\cdot\left(\frac{pp'}{4}\right)^{k/2}\frac{k^{k-\frac12}}{(kp/2)^{k/2-1/2p}(kp'/2)^{k/2-1/2p'}}e^{S(k)-\frac{1}{p}S(kp/2)-\frac{1}{p'}S(kp'/2)} \\
&= C_p^{1/2}\cdot e^{S(k)-\frac{1}{p}S(kp/2)-\frac{1}{p'}S(kp'/2)}.
\end{align*}

Thus, to prove \eqref{e.sharp?2} and thus Theorem \ref{t.3}, it suffices to prove the following proposition.

\begin{proposition} \label{p.Stirling} For any $p\in(1,\infty)\setminus\{2\}$ and any $k\in\N$,
\[ S(k)-\frac{1}{p}S(kp/2)-\frac{1}{p'}S(kp'/2) < 0. \]
Moreover, the limit of this expression as $k\to\infty$ is $0$.
\end{proposition}

\begin{proof} Denote the integrand of $S(x)$ as $s(t,x)$:
\[ s(t,x) = \frac{2\arctan(t/x)}{e^{2\pi t}-1}. \]
Note that $s\in C^\infty((0,\infty)^2)$; the first two $x$ derivatives are as follows:
\[ \frac{\del s}{\del x}(t,x) = -\frac{1}{t^2+x^2} \frac{2t}{e^{2\pi t}-1}, \qquad \frac{\del^2 s}{\del x^2}(t,x) = \frac{x}{(t^2+x^2)^2}\frac{4t}{e^{2\pi t}-1}. \]
Thus, for each $t>0$, $x\mapsto s(t,x)$ is strictly convex on $(0,\infty)$.  In particular, since $\frac{1}{p}+\frac{1}{p'}=1$ and $\frac{1}{p},\frac{1}{p'}\in(0,1)$, and since $p\ne p'$, we have
\[ s(t,x) = s\left(t,\frac{1}{p}\frac{xp}{2}+\frac{1}{p'}\frac{xp'}{2}\right) < \frac{1}{p}s\left(t,\frac{xp}{2}\right)+\frac{1}{p'}s\left(t,\frac{xp'}{2}\right). \]
Since $t\mapsto s(t,x)$ is strictly positive, upon integration this inequality remains strict, and so
\[ S(x) = \int_0^\infty s(t,x)\,dt < \int_0^\infty \frac{1}{p}s\left(t,\frac{xp}{2}\right)\,dt + \int_0^\infty \frac{1}{p'}s\left(t,\frac{xp'}{2}\right)\,dt = \frac{1}{p}S(xp/2)+\frac{1}{p'}S(xp'/2). \]
Taking $x=k\in\N$ proves the first statement of the proposition.

For the second statement, it suffices to show that $\lim_{x\to\infty} S(x)=0$.  As computed above, $\frac{\del s}{\del x}(t,x)<0$, and so $x\mapsto s(t,x)$ is decreasing; in particular, for $x\ge 1$ the integrand is $\le \frac{2\arctan(t)}{e^{2\pi t}-1}$, which is an $L^1(0,\infty)$ function.  Since $\lim_{x\to\infty}\arctan(t/x)=0$ for each fixed $t$, it follows from the Dominated Convergence Theorem that $\lim_{x\to\infty} S(x)=0$, completing the proof. \end{proof}

\begin{remark} \label{r.newS} \begin{itemize} \item[(1)] Note, from (\ref{e.newS}), the statement $\lim_{x\to\infty} S(x)=0$ is (up to a logarithm) precisely the usual statement of Stirling's approximation:
\[ 1= \lim_{x\to\infty} \frac{\Gamma(x)}{\sqrt{\frac{2\pi}{x}}\left(\frac{z}{e}\right)^z} = \lim_{x\to\infty} e^{S(x)}. \]
We include the Dominated Convergence Theorem proof above just for completeness.
\item[(2)] The above computations are only valid for $k>0$.  However, it is easy to check that $\mathcal{R}_{p,\alpha}(g_0,g_0) = 1 < C_p^{1/2}$, since $g_0=1$.
\end{itemize}
\end{remark}

Thus, we have completed the proof of Theorem \ref{t.3}.
Let us also note, for use in the next section, that Proposition \ref{p.Stirling} actually shows that, for each $k\in\N$,
\begin{equation} \label{e.gammaN} \mathcal{R}_{p,\alpha}(g_k,g_k) < C_p^{1/2}. \end{equation}
\section{Monomials and the Ratio $\mathcal{R}_{p,\alpha}(g,h)$ \label{s.local}}
Equation \eqref{e.gammaN} suggests that $C_p^{n/2}$ is actually the supremum of the ratio $\mathcal{R}_{p,\alpha}(g,h)$.  One way to prove that this would be to show that the ratio $\mathcal{R}_{p,\alpha}$ is maximized in some sense on monomials. In this section we will prove a partial result in this vein (Theorem \ref{lemma3}) which has two interesting corollaries.  We begin by stating a result that will be useful in what follows.

\begin{lemma}\label{lemma0}
Let $1<p<\infty$ and $\a>0$. If $f\in\mathcal{H}^1_{p,\alpha}$, then the Taylor series of $f$ centered at $0$ converges to $f$ in the $L^p(\gamma^1_{\a p/2})$ norm.
\end{lemma}
The above result is well-known.  For example, it can be found as \cite[Ex. 5, Ch. 2]{Zhu}.  We next define a nonlinear operator $G^n_{p,\alpha}(h)$ that will be central in proving Theorem \ref{lemma3}, the main result of this section.  Let $1<p<\infty$.  For $h\in L^p(\gamma^n_{\alpha p/2})$, define $G^n_{p,\alpha}(h)$ as
\begin{equation} \label{eq G(h)} G^n_{p,\alpha}(h)(z) \equiv |h(z)|^{p-2}h(z)e^{-\alpha(\frac{p}{2}-1)|z|^2}. \end{equation}
First we note that, if $h\in L^p(\gamma^n_{\alpha p/2})$, then $G^n_{p,\alpha}(h)\in L^{p'}(\gamma^n_{\alpha p'/2})$:
\begin{eqnarray}
\|G^n_{p,\alpha}(h)\|_{p',\alpha} &=& \left(\left(\frac{\alpha p'}{2\pi}\right)^n\int_{\C^n} \left(|h(z)|^{p-2}|h(z)|e^{-\alpha\left(\frac{p}{2}-1\right)|z|^2}\right)^{p'}e^{-\frac{\alpha p'}{2}|z|^2}dz\right)^{1/p'}\nonumber\\
&=&\left(\left(\frac{\alpha p'}{2\pi}\right)^n\left(\frac{2\pi}{\alpha p}\right)^n\int_{\C^n} |h(z)|^{p}\left(\frac{\alpha p}{2\pi}\right)^ne^{-\frac{\alpha p}{2}|z|^2}dz\right)^{1/p'}\nonumber\\
&=&\left(\frac{p'}{p}\right)^{n/p'}\|h\|^{p/p'}_{p,\alpha}.\nonumber
\end{eqnarray}
In fact, the function $G^n_{p,\alpha}(h)$ has been designed to have the property that
\begin{equation} \label{eq G maximizer} \langle h, G^n_{p,\alpha}(h)\rangle_{\alpha} = \left(\frac{2}{p}\right)^n \|h\|_{p,\alpha}^{p}= C_{p}^n\|h\|_{p,\alpha}\|G^n_{p,\alpha}(h)\|_{p',\alpha} . \end{equation}
That is, it gives equality in H\"older's inequality (up to the scale constants required by the dilated Gaussian measures).  Furthermore, for all $g\in L^{p'}(\gamma^n_{\alpha p'/2})$ we have
\begin{equation}\label{eq G maximizer 2}
|\langle h,g\rangle_{\alpha}| = C_p^n\|h\|_{p,\alpha}\|g\|_{p',\alpha} \iff \mbox{$g$ is a constant multiple of $G^n_{p,\alpha}(h)$.}
\end{equation}
By Theorem \ref{t.2}, $G^n_{p,\alpha}(h)$ cannot be holomorphic if $h$ is holomorphic, nonconstant, and $p\neq 2$, but we can consider its projection into holomorphic space:
\begin{lemma}\label{lemma1}
Let $1<p<\infty$, $\a>0$, and $n\in\N$.  For all $f,h\in\H^n_{p,\alpha}\backslash\{0\}$, we have
\begin{equation}\label{reflemmaeq1}
\frac{|\langle f, P^n_\alpha (G^n_{p,\alpha}(h))\rangle_\alpha|}{\|f\|_{p,\alpha}}\leq \frac{|\langle h, P^n_\alpha (G^n_{p,\alpha}(h))\rangle_\alpha|}{\|h\|_{p,\alpha}},
\end{equation}
and equality is achieved if and only if $f$ is a constant multiple of $h$.
\end{lemma}
\begin{proof}
Assuming the premise of the lemma and using H\"older's inequality, we have 
\begin{eqnarray*}
\frac{|\langle f, P^n_\alpha (G^n_{p,\alpha}(h))\rangle_\alpha|}{\|f\|_{p,\alpha}} &=& \frac{|\langle f, G^n_{p,\alpha}(h)\rangle_\alpha|}{\|f\|_{p,\alpha}}\leq C^n_p \|G^n_{p,\alpha}(h)\|_{p',\alpha}\\
&=& \frac{|\langle h, G^n_{p,\alpha}(h)\rangle_\alpha|}{\|h\|_{p,\alpha}} =\frac{|\langle h, P^n_\alpha (G^n_{p,\alpha}(h))\rangle_\alpha|}{\|h\|_{p,\alpha}},
\end{eqnarray*}
where in the second line we used \eqref{eq G maximizer}.  This proves the first part of the lemma.  For the second part, note that by \eqref{eq G maximizer 2} we have equality above if and only if $f$ is a constant multiple of $G^n_{p',\alpha}(G^n_{p,\alpha}(h))$. However,
\begin{eqnarray*}
G^n_{p',\alpha}(G^n_{p,\alpha}(h))(z)&=& |G^n_{p,\alpha}(h)(z)|^{p'-2}G^n_{p,\alpha}(h)(z)e^{-\alpha(\frac{p}{2}-1)|z|^2}\\
&=& |h(z)|^{(p-1)(p'-2)+(p-2)}e^{-\frac{\alpha}{2}((p'-1)(p-2)+(p'-2))|z|^2} h(z)= h(z),
\end{eqnarray*}
as $(p-1)(p'-2)+(p-2)=(p'-1)(p-2)+(p'-2)=0$.  Thus, we have equality in \eqref{reflemmaeq1} if and only if $f$ is a constant multiple of $h$.
\end{proof}
Let $\N_0=\N\cup\{0\}$.  For $\a>0$, $n\in\N$, and a multi-index $\jj=(j_1,j_2,\ldots,j_n)\in\N_0^n$, define the function $\psi_{\a,\jj}$ as
\begin{equation}\label{psidef}
\psi_{\jj,\a}(z) = \sqrt{\frac{\alpha^{|\jj|}}{\jj!}}z^\jj \quad \text{for} \quad z\in\C^n,
\end{equation}
where we use the standard notations $|\jj|=j_1+j_2+\ldots +j_n$, $\jj!=j_1!j_2!\cdot\ldots\cdot j_n!$, and $z^\jj=z_1^{j_1}z_2^{j_2}\cdot\ldots\cdot z_n^{j_n}$.  The set $\{\psi_{\jj,\a}\}_{\jj\in\N_0^n}$ is an orthonormal basis in $\mathcal{H}^n_{2,\a}$.  These functions also have the interesting property that $P^n_\a(G_{p,\alpha}(\psi_{\jj,\a}))$ is a multiple of $\psi_{\jj,\a}$.
\begin{lemma}\label{lemma2}
Let $1<p<\infty$, $\a>0$, and $n\in\N$. For each multi-index $\jj=(j_1,j_2,\ldots,j_n)\in\N_0^n$, the function $\psi_{\jj,\a}$ defined in \eqref{psidef} satisfies
\begin{equation}\label{lemma2eq0}
P^n_{\alpha} G^n_{p,\alpha}(\psi_{\jj,\a}) = \left(\frac{2}{p}\right)^{|\jj|p/2+n}\frac{\prod_{k=1}^n\Gamma(j_k p/2+1)}{\sqrt{\jj!}^{p}}\psi_{\jj,\a}.
\end{equation}
In particular, there exists a constant $K_{\jj,p,\a}$ such that
\begin{equation}\label{lemma2eq1}
P^n_{\alpha} G^n_{p,\alpha}(\psi_{\jj,\a}) = K_{\jj,p,\a}z^\jj.
\end{equation} 
\end{lemma}
\begin{proof}
First note we have for any $\jj\in\N_0^n$
\begin{equation*}
P^n_{\alpha} G^n_{p,\alpha}(\psi_{\jj,\a})(z) = P^n_{\alpha} G^n_{p,\alpha}\left(\prod_{k=1}^n\psi_{j_k,\a}(z_k)\right)=P^n_{\alpha} \prod_{k=1}^n G^1_{p,\alpha}(\psi_{j_k,\a}(z_k)) =  \prod_{k=1}^n P^1_{\alpha} G^1_{p,\alpha}(\psi_{j_k,\a}(z_k)).
\end{equation*}
Thus, the general result will follow if we prove the lemma for $n=1$.  To that end, fix a nonnegative integer $j$.  Using the self-adjointness of $P^1_\alpha$ as well as the fact that $P^1_\alpha$ is the identity on $\mathcal{H}^1_{p,\alpha}$ we have for any other nonnegative integer $k$:
\begin{eqnarray}
\left\langle P^1_\alpha G^1_{p,\alpha}(\psi_{j,\a}), \psi_{k,\a}\right\rangle_\alpha 
&=& \frac{\alpha^{k/2+(p-1)j/2+1}}{\pi \sqrt{k!}\sqrt{j!}^{p-1}}\int_\C |z|^{j(p-2)}z^j(\overline{z})^k e^{-\frac{\alpha p}{2}|z|^2} dz\nonumber\\
&=& \frac{\alpha^{k/2+(p-1)j/2+1}}{\pi \sqrt{k!}\sqrt{j!}^{p-1}}\int_0^\infty r^{j(p-1)+k+1} e^{-\frac{\alpha p}{2}r^2} \left(\int_0^{2\pi} e^{i(j-k)\theta}d\theta\right) dr\nonumber\\
&=& \left(\frac{2\alpha^{pj/2+1}}{\sqrt{j!}^{p}}\int_0^\infty r^{jp+1} e^{-\frac{\alpha p}{2}r^2} dr\right)\delta_{jk}\nonumber\\
&=& \left(\left(\frac{2}{p}\right)^{jp/2+1}\frac{1}{\sqrt{j!}^{p}}\Gamma(jp/2+1)\right)\delta_{jk}.\label{ProjectionOfZm}
\end{eqnarray}
By Lemma \ref{lemma0} the Taylor series of the exponential function converges in $L^p(\gamma^1_{\a p/2})$, and so we use \eqref{ProjectionOfZm} to compute 
\begin{equation*}
P^1_\alpha G^1_{p,\alpha}(\psi_{j,\a})(z) = \sum_{k=0}^\infty \left\langle P^1_\alpha G^1_{p,\alpha}(\psi_{j,\a}), \psi_{k,\a}\right\rangle_\alpha \psi_{k,\a}(z) 
= \left(\frac{2}{p}\right)^{jp/2+1}\frac{\Gamma(jp/2+1)}{\sqrt{j!}^{p}}\psi_{j,\a}(z),
\end{equation*}
proving \eqref{lemma2eq0}.  Equation \eqref{lemma2eq1} follows from \eqref{lemma2eq0} as $z^\jj$ is a constant multiple of $\psi_{\jj,\a}$.
\end{proof}
We now can prove our partial result on the ratio $\mathcal{R}_{p,\alpha}$ taking on maximal values on monomials:
\begin{theorem}\label{lemma3}
Let $1<p<\infty$, $\a>0$, and $n\in\N$.  For any $f\in\H^n_{p,\alpha}\backslash\{0\}$ and multi-index $\jj\in\N^n$, we have
\begin{equation}\label{lemma3eq0}
\frac{|\langle f, z^\jj\rangle_\alpha|}{\|f\|_{p,\alpha}} \leq \frac{|\langle z^\jj, z^\jj\rangle_\alpha|}{\|z^\jj\|_{p,\alpha}}, 
\end{equation}
or, equivalently,
\begin{equation}\label{lemma3eq-1}
\mathcal{R}_{p,\a}(f,z^\jj) \leq \mathcal{R}_{p,\a}(z^\jj,z^\jj). 
\end{equation}
Furthermore, both inequalities \eqref{lemma3eq0} and \eqref{lemma3eq-1} are sharp and are equality if and only if $f(z)$ is a constant multiple of $z^\jj$.
\end{theorem}
\begin{proof}
Assume the premise of the theorem.  First note that \eqref{lemma3eq0} and \eqref{lemma3eq-1} are equivalent, as \eqref{lemma3eq-1} is just \eqref{lemma3eq0} with both sides of the inequality divided by $\|z^\jj\|_{p',\a}$.  Thus, to prove the rest of the lemma, it suffices to consider only \eqref{lemma3eq0}. To that end, we apply Lemma \ref{lemma1} to the functions $f$ and $\psi_{\jj,\a}$:
\begin{equation}
\frac{|\langle f, P^n_\alpha (G^n_{p',\alpha}(\psi_{\jj,\a}))\rangle_\alpha|}{\|f\|_{p,\alpha}} \leq \frac{|\langle \psi_{\jj,\a}, P^n_\alpha (G^n_{p',\alpha}(\psi_{\jj,\a})) \rangle_\alpha|}{\|\psi_{\jj,\a}\|_{p,\alpha}}= \frac{|\langle z^\jj, P^n_\alpha (G^n_{p',\alpha}(\psi_{\jj,\a})) \rangle_\alpha|}{\|z^\jj\|_{p,\alpha}}.\label{lemma3eq1}
\end{equation}
By Lemma \ref{lemma2}, dividing both sides of the inequality above by $K_{\jj,p,\a}$ yields \eqref{lemma3eq0}.  Note that by Lemma \ref{lemma1} we have equality in \eqref{lemma3eq1} if and only if $f$ is a constant multiple of $\psi_{\jj,\a}$, which is equivalent to $f(z)$ being a constant multiple of $z^\jj$, as desired.
\end{proof}
\begin{remark}
By symmetry of $\mathcal{R}_{p,\a}$, Theorem \ref{lemma3} also implies that, for any fixed $1<p<\infty$, $\a>0$, $n\in\N$, and $\jj\in\N^n$, we have for all $g\in\mathcal{H}^n_{p',\a}$
\[\mathcal{R}_{p,\a}(z^\jj,g)\leq \mathcal{R}_{p,\a}(z^\jj,z^\jj).\]
\end{remark}

Theorem \ref{lemma3} is a weak form of the fully conjectured theorem: that $\sup_{f,g}\mathcal{R}_{p,\a}(f,g)=C_p^{n/2}$ (or equivalently $\sup_{f,g}\mathcal{R}_{p,\a}(f,g)\le C_p^{n/2}$).  Indeed, Section \ref{s.proof of t.3} computes that $\mathcal{R}_{p,\a}(z^\jj,z^\jj) < C_p^{n/2}$ for all $\jj$ (and that $C_p^{n/2}$ is the limit as $\jj\to\infty$), and so we have shown that
\begin{equation} \label{e.max.each.variable} \mathcal{R}_{p,\a}(f,z^\jj) < C_p^{n/2} \quad\text{ for all } \; \jj\in\N_0^n. \end{equation}
Thus, restricting one of the variables in $\mathcal{R}_{p,\alpha}$ to range through the space of monomials $z^\jj$ gives the global maximum result, proving Theorem \ref{t.4.new}.  This is, of course, a far cry from the desired theorem, but it is suggestive.  We also have two interesting corollaries from Theorem \ref{lemma3}. The first holds for all finite complex dimensions $n$, while the second focuses on Taylor coefficients when $n=1$.  Both corollaries involve ``projecting'' a function onto monomials.  More specifically, fix a complex dimension $n$.  For any multi-index $\jj\in\N^n$, define the map $P_{\jj,\alpha}:\H^n_{p,\alpha}\to\mathrm{Span}\{z^\jj\}\subseteq \H^n_{p,\alpha}$ as
\[P_{\jj,\alpha}(f) = \langle f, \psi_{\jj,\a}\rangle_\alpha \psi_{\jj,\a}.\]
By H\"older's inequality, this mapping is continuous.  Heuristically, $P_{\jj,\alpha}(h)$ is ``projecting'' $h$ onto the $\jj^{th}$ monomial.  The corollary below gives us more reason to call $P_{\jj,\alpha}$ a ``projection,'' despite the lack of a Hilbert space structure to the space $\H^n_{p,\alpha}$.
\begin{corollary}\label{cor53}
Let $1<p<\infty$, $\a>0$, and $n\in\N$.  The mapping $P_{\jj,\alpha}:\H^n_{p,\alpha}\to\H^n_{p,\alpha}$ has norm $1$.  That is, if $f\in\H^n_{p,\alpha}$, then
\begin{equation}\label{cor53eq0}
\|\langle f, \psi_{\jj,\a}\rangle_\alpha \psi_{\jj,\a}\|_{p,\alpha}\leq \|f\|_{p,\alpha}.
\end{equation}
Furthermore, we have equality if and only if $f(z)$ is a constant multiple of $z^\jj$.
\end{corollary}
\begin{proof}
Assume the premise of the corollary.  If $f$ is the zero function, then the result follows.  So assume $f$ is not the $0$ function.  As $\psi_{\jj,\a}$ is a scalar multiple of $z^\jj$, by Theorem \ref{lemma3}, we have
\[\frac{|\langle f, \psi_{\jj,\a}\rangle_\alpha|}{\|f\|_{p,\alpha}} \leq \frac{|\langle \psi_{\jj,\a}, \psi_{\jj,\a}\rangle_\alpha|}{\|\psi_{\jj,\a}\|_{p,\alpha}}= \frac{1}{\|\psi_{\jj,\a}\|_{p,\alpha}}.\]
Multiplying the above by $\|\psi_{\jj,\a}\|_{p,\alpha}\|f\|_{p,\alpha}$ yields \eqref{cor53eq0}.  The equality condition again follows from the equality condition in Theorem \ref{lemma3}.
\end{proof}
Now we can use Corollary \ref{cor53} to prove another corollary on Taylor coefficients in the case where $n=1$.
\begin{corollary}\label{cor5n=1}
Let $1<p<\infty$ and $\a>0$. Let $f\in\mathcal{H}^1_{p,\a}$ and with the Taylor series $f(z)=\sum_{k=0}^\infty a_k z^k$. Then for any $j\in\N_0$, 
\begin{equation}\label{cor5n=1eq1}
\|a_j z^j\|_{p,\alpha}\leq \left\|\sum_{k=0}^\infty a_k z^k\right\|_{p,\alpha}=\|f\|_{p,\alpha},
\end{equation}
and
\begin{equation}\label{cor5n=1eq2}
|a_j|\leq \frac{\left(\frac{\alpha p}{2}\right)^{j/2}}{\Gamma\left(\frac{j p}{2}+1\right)^{1/p}}\|f\|_{p,\alpha}.
\end{equation}
In either inequality above, we have equality if and only if $f(z)$ is a constant multiple of $z^j$.
\end{corollary} 
\begin{proof}
Assume the premise of the corollary.  We first prove \eqref{cor5n=1eq1}.  By Lemma \ref{lemma0}, the Taylor series of $f(z)$ converges in norm, and so for any fixed $j\in\N$, we have (using $w$ as the integration variable here)
\begin{equation}\label{cor5n=1eq3}\langle f,\psi_{j,\a}\rangle_\a\psi_{j,\a} = \left\langle \sum_{k=0}^\infty a_k w^k,\psi_{j,\a}(w) \right\rangle_\a\psi_{j,\a}(z)=a_j \langle w^j, \psi_{j,\a}(w)\rangle_\a\psi_{j,\a}(z)=a_j z^j.\end{equation}
Inequality \eqref{cor5n=1eq1} as well as the equality statement now follows from Corollary \ref{cor53} and \eqref{cor5n=1eq3}.  
To prove \eqref{cor5n=1eq2}, first note that 
\begin{eqnarray*}
\|z^j\|_{p,\alpha} &=& \left(\int_\C|z|^{j p}\left(\frac{\alpha p}{2\pi}\right)e^{-\frac{\alpha p}{2}|z|^2} dz\right)^{1/p}\\
&=& \left(\left(\frac{2}{\alpha p}\right)^{j p/2}\int_0^\infty \left(\frac{\alpha p}{2}r^2\right)^{j p/2}e^{-\frac{\alpha p}{2}r^2} \alpha pr dr\right)^{1/p}\\
&=& \left(\left(\frac{2}{\alpha p}\right)^{jp/2}\int_0^\infty u^{mp/2}e^{-u}du\right)^{1/p}\\
&=& \left(\frac{2}{\alpha p}\right)^{j/2}\Gamma\left(j p/2+1\right)^{1/p}.
\end{eqnarray*}
Thus dividing \eqref{cor5n=1eq1} by $\|z^j\|_{p,\a}$ yields \eqref{cor5n=1eq2}, proving the corollary.
\end{proof}
\section{The Conjectured Sharp Constant and Quadratic Exponentials}
Let $\mathcal{M}=\{z^N: \mbox{$N$ is a nonnegative integer}\}$.  We previously showed in \eqref{e.gammaN} that, for dimension $n=1$,
$$\sup_{g,h\in\mathcal{M}} \mathcal{R}_{p,\alpha}(g,h) = \sqrt{C_{p}}.$$
To strengthen our evidence for $\sqrt{C_p}$ being the sharp constant for $n=1$ we will prove a similar equality over a larger family of functions that is more natural to the problem's setting.  To that end, we consider functions $g$ in the family
\begin{equation*}
\mathcal{G}=\left\{e^{p(z)}: \mbox{$p(z)$ is a polynomial} \right\}.
\end{equation*}
These functions are holomorphic but not generally in $\H^1_{p,\a}$.  In particular, if $g(z)=e^{p(z)}\in \mathcal{H}^n_{p,\alpha}$, then by \eqref{e.pointwise1} the function $e^{\Re(p(z))-\alpha |z|^2/2}$ is bounded on $\C$ by $\|g\|_{p,\alpha}$, implying that the degree of $p(z)$ must be no larger than 2.  In fact, if $p(z)$ is quadratic, then its leading coefficient must satisfy a certain inequality, as proven in the next lemma:

\begin{lemma}\label{gaussianpnorma}
Let $\alpha>0$ and $a, c\in\C$, and define $g(z)=\exp\left(\alpha a z+ \frac{\alpha}{2}cz^2\right)$.  If $1\leq p<\infty$, then $f(z)\in \mathcal{H}^1_{p,\alpha}$ if and only if $|c|<1$.  In this case, we have 
$$\|g\|_{p,\alpha}=\frac{1}{(1-|c|^2)^{1/2p}}\exp\left(\frac{\alpha}{2}\frac{|a|^2+\Re(\overline{c}a^2)}{1-|c|^2}\right).$$
\end{lemma}
\begin{proof}
Define $g(z)=e^{\left(\alpha a z+\frac{\alpha}{2}cz^2\right)}$ for some constants $a, c\in\C$.  Define a matrix $A$ and a vector $v$ as
\[ A=\begin{bmatrix}1-\Re(c)&\Im(c)\\\Im(c)&1+\Re(c)\end{bmatrix}, \qquad v = \begin{bmatrix}\Re(a)\\-\Im(a)\end{bmatrix}. \]
Note that $A$ is positive definite $\iff$ $1-\Re(c)>0$ and $1-\Re(c)^2-\Im(c)^2>0$ $\iff$ $|c|^2<1$ $\iff$ $|c|<1$.  Now we have
\begin{equation*}
\|g\|^p_{p,\alpha}= \int_\C |e^{\left(\alpha a z+\frac{\alpha}{2}cz^2\right)}|^p\left(\frac{\alpha p}{2\pi}\right)e^{-\frac{\alpha p}{2}|z|^2}dz
= \left(\frac{\alpha p}{2\pi}\right)\int_{\R^{2}} e^{\left(-\left(x,\frac{\alpha p}{2}Ax\right)+2\left(\frac{\alpha p}{2}v,x\right)\right)}dx.
\end{equation*}
By Lemma \ref{l.Gauss.int}, the above integral converges if and only if $A$ is positive definite which holds if and only if $|c|<1$, proving the first part of the lemma.  Assuming $|c|<1$, we have
\begin{equation}
\|g\|^p_{p,\alpha}
= \left(\frac{\alpha p}{2\pi}\right)\int_{\R^{2}} e^{\left(-\left(x,\frac{\alpha p}{2}Ax\right)+2\left(\frac{\alpha p}{2}v,x\right)\right)}dx
=\frac{1}{\sqrt{1-|c|^2}}\exp\left(\frac{\alpha p}{2}v,A^{-1}v\right).\label{gaussianpnorma1}
\end{equation}
One can compute the inner product $(v,A^{-1}v)$ as
\begin{equation}
(v,A^{-1}v) = \frac{1}{1-|c|^2}[\Re(a)\,\,-\Im(a)]\begin{bmatrix}1+\Re(c)&-\Im(c)\\-\Im(c)&1-\Re(c)\end{bmatrix}\begin{bmatrix}\Re(a)\\-\Im(a)\end{bmatrix}
=\frac{|a|^2+\Re(\overline{c}a^2)}{1-|c|^2}\label{gaussianpnorma2}.
\end{equation}
Substituting \eqref{gaussianpnorma2} into \eqref{gaussianpnorma1} and taking $p^{th}$ roots gives the desired result.
\end{proof}

In light of the above lemma, we define a subset $\mathcal{G}_\alpha$ of $\mathcal{G}$ as the appropriate quadratic exponential functions
\begin{equation}\label{Galphadef}
\mathcal{G}_\alpha=\mathcal{G}\cap \mathcal{H}^n_{2,\alpha}=\left\{\exp\left(\alpha a z+\frac{\alpha}{2}c z^2\right): a, c\in\C, |c|< 1\right\}.
\end{equation}
We chose $p=2$ in $\mathcal{H}^n_{2,\alpha}$ above, but note that by Lemma \ref{gaussianpnorma} and the discussion that preceded it, $\mathcal{G}_\alpha=\mathcal{G}\cap \mathcal{H}^n_{p,\alpha}$ for any $1\leq p <\infty$.  We can now state the main theorem of this section.

\begin{theorem}\label{gaussianprop}
Let $\mathcal{G}_\alpha$ be defined as in \eqref{Galphadef} and let $1 < p<\infty$.  Then 
\begin{equation*}
\sup_{g,h\in\mathcal{G}_\alpha} \mathcal{R}_{p,\alpha}(g,h) = \sqrt{C_{p}}.
\end{equation*}
\end{theorem}

The function $\mathcal{R}_{p,\a}$ is a ``Gaussian kernel'': it is defined on an $L^p\times L^{p'}$ space with a Gaussian weight.  It does not quite fit the conditions of \cite{Lieb} (which asserts that, in a similar setup, Gaussian kernels have only Gaussian maximizers), but given that landmark paper it is very natural to think that the maximum should occur on quadratic exponential functions in this case as well.  Thus, Theorem \ref{gaussianprop} provides more highly suggestive evidence that maximum really is $\sqrt{C_p}$, at least in the $n=1$ dimensional case.

We will break the proof of Theorem \ref{gaussianprop} into multiple lemmas.  We begin by computing the ratio $\mathcal{R}_{p,\alpha}(g,h)$ for functions $g,h\in\mathcal{G}_\alpha$:

\begin{lemma}\label{gaussianpairinga}
Let $g,h\in\mathcal{G}_\alpha$ be arbitrary with $g(z)=\exp\left(\alpha az+\frac{\alpha}{2}c z^2\right)$ and $h(z)=\exp\left(\alpha bz+\frac{\alpha}{2}d z^2\right)$.  Then
\begin{eqnarray}
\lefteqn{\mathcal{R}_{p,\alpha}(g,h)=}\nonumber\\
& & \sqrt{\frac{(1-|c|^2)^{1/p}(1-|d|^2)^{1/p'}}{|1-c\overline{d}|}}\exp\left(\frac{\alpha}{2}\Re\left[\frac{2a\overline{b}+a^2\overline{d}+\overline{b^2}c}{1-c\overline{d}}-\frac{|a|^2+a^2\overline{c}}{1-|c|^2}-\frac{|b|^2+\overline{b^2}d}{1-|d|^2}\right]\right).\label{gaussianparingaeq1}
\end{eqnarray}
\end{lemma}
\begin{proof}
We first compute $|\langle g,h\rangle_\alpha|$.  This computation here is very similar although a bit more complicated to the proof of Lemma \ref{gaussianpnorma}.  Define a matrix $B$ and a vector $w$ as
\begin{equation*}
B=\begin{bmatrix}2-(c+\overline{d})& -ic+i\overline{d}\\-ic+i\overline{d}&2+(c+\overline{d})\end{bmatrix}, \qquad w=\begin{bmatrix} a+\overline{b} \\ i(a-\overline{b})\end{bmatrix}.
\end{equation*}
As $|c|, |d|<1$ it is easy to see that $B$ has a positive definite real part.  One can compute that
\begin{equation}
\langle g,h\rangle_{\alpha}
= \frac{\alpha}{\pi}\int_{\R^2} e^{-(x,\frac{\alpha}{2}Bx)+2\left(\frac{\alpha}{2}w,x\right)}dx\nonumber\\
= \frac{2}{\sqrt{\det(B)}}= \frac{1}{\sqrt{1-c\overline{d}}}e^{\frac{\alpha}{2}\left(w,B^{-1} w\right)}\label{gaussianpairinga1}.
\end{equation}
We compute $\left(w,B^{-1} w\right)$ as
\begin{equation}
\left(w,B^{-1} w\right)
= \frac{8a\overline{b}+4a^2\overline{d}+4\overline{b^2}c}{4(1-c\overline{d})}= \frac{2a\overline{b}+a^2\overline{d}+\overline{b^2}c}{1-c\overline{d}}.\label{gaussianpairinga2}
\end{equation}
Combining \eqref{gaussianpairinga1} and \eqref{gaussianpairinga2} with Lemma \ref{gaussianpnorma} we have
\begin{eqnarray*}
\mathcal{R}_{p,\alpha}(g,h)&=&\frac{|\langle g, h\rangle_\alpha|}{\|g\|_{p,\alpha}\|h\|_{p',\alpha}}\\
&=&\sqrt{\frac{(1-|c|^2)^{1/p}(1-|d|^2)^{1/p'}}{|1-c\overline{d}|}}\exp\left(\frac{\alpha}{2}\Re\left[\frac{2a\overline{b}+a^2\overline{d}+\overline{b^2}c}{1-c\overline{d}}-\frac{|a|^2+a^2\overline{c}}{1-|c|^2}-\frac{|b|^2+\overline{b^2}d}{1-|d|^2}\right]\right),
\end{eqnarray*}
completing the proof.
\end{proof}
The next step is to understand the exponential term in \eqref{gaussianparingaeq1} and to show that it is never greater than $1$.  To that end,  we have the following lemma:
\begin{lemma}\label{criticalpointlemma1}
Fix $b, c, d\in\C$ where $|c|<1$ and $|d|<1$.  Define the function $f:\R^2\to\R$ as
\begin{equation}
f(x,y) = \Re\left[\frac{2(x+iy)\overline{b}+(x+iy)^2\overline{d}+\overline{b^2}c}{1-c\overline{d}}-\frac{|x+iy|^2+(x+iy)^2\overline{c}}{1-|c|^2}-\frac{|b|^2+\overline{b^2}d}{1-|d|^2}\right].
\end{equation}
Then $f$ has a unique critical point $(x_0,y_0)$ that satisfies
\begin{equation}\label{criticalpointeq2a}
x_0+iy_0 = \frac{\overline{b}(d-c)+b(1-c\overline{d})}{1-|d|^2}.
\end{equation}
Furthermore, $\sup_{(x,y)\in\R^2} f(x,y)=f(x_0,y_0)=0$.
\end{lemma}
\begin{proof}
We will prove this lemma using calculus. First we compute the partial derivatives of $f$:
\begin{eqnarray}
\frac{\partial f}{\partial x} &=& \Re\left[\frac{2\overline{b}+2(x+iy)\overline{d}}{1-c\overline{d}}-\frac{2(x+iy)\overline{c}}{1-|c|^2}\right]-\frac{2x}{1-|c|^2},\nonumber\\
\frac{\partial f}{\partial y} &=& \Re\left[\frac{2i\overline{b}+2i(x+iy)\overline{d}}{1-c\overline{d}}-\frac{2i(x+iy)\overline{c}}{1-|c|^2}\right]-\frac{2y}{1-|c|^2}\nonumber\\
& &= -\Im\left[\frac{2\overline{b}+2(x+iy)\overline{d}}{1-c\overline{d}}-\frac{2(x+iy)\overline{c}}{1-|c|^2}\right]-\frac{2y}{1-|c|^2}\nonumber
\end{eqnarray}
We can rewrite the above as
\begin{equation}\label{criticalpointeq0a}
\frac{1}{2}\left(\frac{\partial f}{\partial x}+i\frac{\partial f}{\partial y}\right)= \frac{b+(x-iy)d}{1-\overline{c}d}-\frac{(x-iy)c+(x+iy)}{1-|c|^2}.
\end{equation}
Suppose that $(x_0,y_0)$ is a critical point.  We then can set the above equation to $0$ and solve for $b$, yielding
\begin{equation}\label{criticalpointeq1a}
b = \frac{(x_0-iy_0)(c-d)+(x_0+iy_0)(1-\overline{c}d)}{1-|c|^2}.
\end{equation}
Using \eqref{criticalpointeq1a} one can verify that \eqref{criticalpointeq2a} holds.
Furthermore, one can rewrite \eqref{criticalpointeq2a} as
\begin{equation}\label{criticalpointeq3a}
0= \frac{(x_0+iy_0)+\overline{b}c}{1-c\overline{d}}-\frac{b+\overline{b}d}{1-|d|^2}.
\end{equation}
Setting \eqref{criticalpointeq0a} equal to $0$, conjugating, and multiplying the result by $x_0+iy_0$, as well as multiplying \eqref{criticalpointeq3a} by $\overline{b}$ yield the two equations
\begin{eqnarray*}
0 &=& \frac{(x_0+iy_0)\overline{b}+(x_0+iy_0)^2\overline{d}}{1-c\overline{d}}-\frac{(x_0+iy_0)^2\overline{c}+|x_0+iy_0|^2}{1-|c|^2},\\
0 &=& \frac{(x_0+iy_0)\overline{b}+\overline{b}^2c}{1-c\overline{d}}-\frac{|b|^2+\overline{b}^2d}{1-|d|^2}.
\end{eqnarray*}
Adding the two equations above yields $f(x_0,y_0)=0$.  To complete the proof, we need to show that $f$ takes on a global maximum at $(x_0,y_0)$. To that end, we compute the Hessian of $f(x,y)$.  First note that from \eqref{criticalpointeq0a} we have
\begin{eqnarray*}
\frac{1}{2}\frac{\partial}{\partial x}\left(\frac{\partial f}{\partial x}+i\frac{\partial f}{\partial y}\right)&=& \frac{d}{1-\overline{c}d}-\frac{1+c}{1-|c|^2}\\
\frac{1}{2}\frac{\partial}{\partial y}\left(\frac{\partial f}{\partial x}+i\frac{\partial f}{\partial y}\right)&=& i\left(\frac{-d}{1-\overline{c}d}-\frac{1-c}{1-|c|^2}\right)
\end{eqnarray*}
Thus, the Hessian matrix $H$ of $f$ at any point $(x,y)$ is given by
\begin{equation}\label{hessianmatrix}
H = 4\begin{bmatrix}\Re\left(\frac{d}{1-\overline{c}d}-\frac{1+c}{1-|c|^2}\right) & \Im\left(\frac{d}{1-\overline{c}d}-\frac{c}{1-|c|^2}\right)\\\Im\left(\frac{d}{1-\overline{c}d}-\frac{c}{1-|c|^2}\right)& \Re\left(\frac{-d}{1-\overline{c}d}-\frac{1-c}{1-|c|^2}\right) \end{bmatrix}
\end{equation}
Thus, to show that $f(x_0,y_0)$ is a global maximum, it suffices to show that $H$ is negative definite.  Let $h_{11}$ be the $(1,1)$ entry of $H$, and define a function $g:\mathbb{D}\to\R$ (where $\mathbb{D}=\{z\in\C: |z|\leq 1\}$) as
\begin{equation*}
g(z) = \Re\left(\frac{z}{1-\overline{c}z}-\frac{1+c}{1-|c|^2}\right).
\end{equation*}
Note that $h_{11}=4g(d)$.  As $g$ is the real part of a holomorphic function, it is harmonic and thus takes its maximum and minimum values on $\partial\mathbb{D}$, that is, the unit circle.  By assumption $d$ is in the interior of $\mathbb{D}$, so we must have
\begin{equation}\label{hessianeq1}
\frac{1}{4}h_{11}=g(d) < \max_{\theta\in[-\pi,\pi]} g(e^{i\theta}).
\end{equation}
One can compute that
$$\frac{d}{d\theta} g(e^{i\theta})=\frac{d}{d\theta}\Re\left(\frac{1}{e^{-i\theta}-\overline{c}}-\frac{1+c}{1-|c|^2}\right)=\Re\left(i\frac{e^{-i\theta}}{(e^{-i\theta}-\overline{c})^2}\right)=\Im\left(\frac{e^{-i\theta}}{(e^{-i\theta}-\overline{c})^2}\right).$$
Thus, we have a critical point of $g(e^{i\theta})$ exactly when $\Im\left(\frac{e^{-i\theta}}{(e^{-i\theta}-\overline{c})^2}\right)=0$, which is equivalent to
\begin{equation*}
\frac{e^{-i\theta}}{(e^{-i\theta}-\overline{c})^2} = \pm \left|\frac{e^{-i\theta}}{(e^{-i\theta}-\overline{c})^2}\right|.
\end{equation*}
Using the fact that $\left|\frac{e^{-i\theta}}{(e^{-i\theta}-\overline{c})^2}\right|=\frac{1}{(e^{-i\theta}-\overline{c})(e^{i\theta}-c)}$ we can solve the above equation to find two critical points: $\frac{c+1}{1+\overline{c}}$ and $\frac{c-1}{1-\overline{c}}$. 
  As the circle $\partial \mathbb{D}$ is closed, the maximum of $g$ must occur on one of these points and the minimum on the other.  A straightforward calculation yields
$$g\left(\frac{c+1}{1+\overline{c}}\right)=0,\, g\left(\frac{c-1}{1-\overline{c}}\right)=-\frac{2}{1-|c|^2}.$$
Hence, from \eqref{hessianeq1} we have
\begin{equation}\label{hessianeq2}
h_{11}=4g(d) < 4\max_{\theta\in[-\pi,\pi]} g(e^{i\theta}) = 0.
\end{equation}
Next we show that $\det(H)$ is positive.  To that end, one can compute that
\begin{equation*}
\frac{1}{4}\det(H) 
= \left(\frac{1}{1-|c|^2}\right)^2-\left|\frac{d}{1-\overline{c}d}-\frac{c}{1-|c|^2}\right|^2
\end{equation*}
As $|d|<1$, by the Maximum Modulus Principle, we thus have
\begin{equation}
\frac{1}{4}\det(H) > \left(\frac{1}{1-|c|^2}\right)^2-\max_{\theta\in[-\pi,\pi)}\left|\frac{e^{i\theta}}{1-\overline{c}e^{i\theta}}-\frac{c}{1-|c|^2}\right|^2\label{hessianeq4}.
\end{equation}
Note that for any real number $\theta$, we have
\begin{equation*}
\left|\frac{e^{i\theta}}{1-\overline{c}e^{i\theta}}-\frac{c}{1-|c|^2}\right|^2 
=\left(\frac{1}{1-|c|^2}\right)^2\left|\frac{e^{i\theta}-c}{e^{-i\theta}-\overline{c}}\right|^2=\left(\frac{1}{1-|c|^2}\right)^2.
\end{equation*}
Plugging the above into \eqref{hessianeq4} yields
\begin{equation}\label{hessianeq5}
\frac{1}{4}\det(H) > \left(\frac{1}{1-|c|^2}\right)^2-\max_{\theta\in[-\pi,\pi)}\left|\frac{e^{i\theta}}{1-\overline{c}e^{i\theta}}-\frac{c}{1-|c|^2}\right|^2 = 0.
\end{equation}
By \eqref{hessianeq2} and \eqref{hessianeq5}, the matrix $H$ is negative definite, as desired.
\end{proof}
Now we are in a position to prove Theorem \ref{gaussianprop}.
\begin{proof}[Proof of Theorem \ref{gaussianprop}]
By Lemmas \ref{gaussianpairinga} and \ref{criticalpointlemma1}, we have
\begin{equation}
\sup_{g,h\in\mathcal{G}_\alpha} \mathcal{R}_{p,\alpha}(g,h) = \sup_{g,h\in\tilde{\mathcal{G}_\alpha}} \mathcal{R}_{p,\alpha}(g,h) = \sup_{c,d\in \C, |c|,|d|<1} \sqrt{\frac{(1-|c|^2)^{1/p}(1-|d|^2)^{1/p'}}{|1-c\overline{d}|}}.
\end{equation}
Note that, by the reverse triangle inequality, for $|c|,|d|<1$ we have
$$\sqrt{\frac{(1-|c|^2)^{1/p}(1-|d|^2)^{1/p'}}{|1-c\overline{d}|}}\leq \sqrt{\frac{(1-|c|^2)^{1/p}(1-|d|^2)^{1/p'}}{1-|c||d|}}.$$
Hence, we have
\begin{equation}\label{gaussianpropeq1}
\sup_{g,h\in\mathcal{G}_\alpha} \mathcal{R}_{p,\alpha}(g,h) = \sup_{x,y\in [0,1)} \sqrt{\frac{(1-x^2)^{1/p}(1-y^2)^{1/p'}}{1-xy}}.
\end{equation}
Fix $x\in [0,1)$ and consider the function $h_x(y):[0,1)\to[0,\infty)$ as
\begin{equation}
h_x(y) = \frac{(1-x^2)^{1/p}(1-y^2)^{1/p'}}{1-xy}.
\end{equation}
Then $h_x$ is differentiable and 
$$h_x'(y)
\frac{(1-x^2)^{1/p}(1-y^2)^{1/p'-1}}{p'2y^2(1-xy)^2}\left(\frac{2y^2+p'(1-y^2)}{2y^2}x-1\right).$$
Thus, $y_0$ is a critical point of $h_x$ if and only if
\begin{equation}\label{ycritical}
x=x(y)=\frac{2y}{2y^2+p'(1-y^2)}=\frac{2y}{(2-p')y^2+p'}.
\end{equation}
It is easy to see that $x(0)=0$, $x(1)=1$, and that $x'(y)>0$ (the fact that $p'>1$ is important here).  Thus, $x(y)$ is an invertible function that maps $[0,1]$ onto $[0,1]$.  Hence, $h_x$ has exactly one critical point $y_x$ and $y_x$ satisfies \eqref{ycritical}.  If $y'<y_x$, then $x(y')<x$ which implies $ h_x'(y')>0$.
Similarly, if $y'>y_x$, then $h_x'(y')<0$, proving that $y_x$ is a local maximum for $h_x$.  Since $y_x$ is a unique critical point, it must be a global maximum for $h_x$.  Also, by the invertibility of $x(y)$, for every $y\in[0,1)$, there exists a unique $x\in [0,1)$ for which $y=y_x$ (namely, $x(y)$).  Thus, we can rewrite \eqref{gaussianpropeq1} as
\begin{equation}\label{gaussianpropeq2}
\sup_{g,h\in\mathcal{G}_\alpha} \mathcal{R}_{p,\alpha}(g,h) = \sup_{x\in [0,1)} \sqrt{h_x(y_x)}=\sup_{y\in [0,1)} \sqrt{h_{x(y)}(y)}.
\end{equation} 
A computation shows that 
\begin{equation}\label{gaussianpropeq3}
h_{x(y)}(y) = \frac{1}{p'}((p')^2-(2-p')^2y^2)^{1/p}((2-p')y^2+p')^{1-2/p}.
\end{equation}
Note that, while $h_{x(y)}(y)$ is not formally defined at $y=1$, it can be continuously extended to 1.  In fact, define $g:[0,1]\to[0,\infty)$ as
$$g(y) =\frac{1}{p'}((p')^2-(2-p')^2y^2)^{1/p}((2-p')y^2+p')^{1-2/p}, $$
and note that $g$ is continuous on $[0,1]$ and differentiable on $(0,1)$.  For $y\in (0,1)$,
\begin{equation*}
g'(y) = \frac{2y}{p'}(2-p')^2((p')^2-(2-p')^2y^2)^{(p'-1)/p'-1}((2-p')y^2+p')^{(2-p')/p'}>0,
\end{equation*}
proving that $g$ is increasing on its domain.  Hence, combining equations \eqref{gaussianpropeq2} and \eqref{gaussianpropeq3} we have  
\begin{eqnarray*}
\sup_{g,h\in\mathcal{G}_\alpha} \mathcal{R}_{p,\alpha}(g,h) &=& \sup_{y\in [0,1)} \sqrt{h_{x(y)}(y)} = \sup_{y\in [0,1]} \sqrt{g(y)} = g(1)\\
&=& \sqrt{\frac{((p')^2-(2-p')^2)^{1/p}((2-p')+p')^{1-2/p}}{p'}}=\sqrt{C_p}
\end{eqnarray*}
completing the proof.
\end{proof}
\noindent{\bf Acknowledgement}. We would like to thank the referee of an earlier version of this manuscript for important insights: in particular the idea behind Theorem \ref{lemma3}.
\bibliography{Fock-kernel-ref}
\bibliographystyle{acm}

\end{document}